\documentclass[11pt, reqno]{amsart}
\usepackage{amssymb,latexsym,amsmath,amsfonts, eucal}
\input{amssym.def}
\date{}
\def\A{{\mathcal A}}
\def\C{{\mathbb C}}
\def\D{{\mathcal D}}
\def\O{{\mathcal O}}
\def\H{{\mathcal H}}

\def\N{\mathbb{N}}

\def\H{\mathcal H}

\def\D{\mathbb D}

\def\ra{\rightarrow}
\def\ov{\overline}
\def\lo{\longrightarrow}

\def\d{\sum}

\def\beqa{\begin{eqnarray*}}
\def\eeqa{\end{eqnarray*}}

 \newtheorem{thm}{Theorem}[section]
 \newtheorem{cor}[thm]{Corollary}
 \newtheorem{lem}[thm]{Lemma}
 \newtheorem{prop}[thm]{Proposition}
 \theoremstyle{definition}
 \newtheorem{defn}[thm]{Definition}
 \newtheorem{rem}[thm]{Remark}
 
 \numberwithin{equation}{section}

\begin{document}
\title[Hilbert modules]{Unitary invariants for Hilbert modules of finite rank }

\author[Biswas]{Shibananda Biswas}

\author[Misra]{Gadadhar Misra}
\address[Shibananda Biswas and Gadadhar Misra]{Department of Mathematics, Indian Institute of Science, Banaglore 560012}

\author[Putinar]{Mihai Putinar}
\address[Mihai Putinar]{Department of Mathematics , University of California, Santa Barbara, CA 93106}

\email[S. Biswas]{shibu@math.iisc.ernet.in} \email[G. Misra]{gm@math.iisc.ernet.in} \email[M.
Putinar]{mputinar@math.ucsb.edu}

\thanks{Financial support for the work of S. Biswas was provided in the
form of a Research Fellowship of the Indian Statistical Institute.  The work of G. Misra was  supported in part
by a grant from the Department of Science and Technology, India. The work of  M. Putinar was supported in part
by a grant from National Science Foundation, US }



\begin{abstract}
A refined notion of curvature for a linear system of Hermitian vector spaces, in the sense of Grothendieck, leads to the unitary classification of a large class of analytic Hilbert modules. Specifically, we study Hilbert sub-modules, for which the localizations are of finite (but not constant) dimension, of an analytic function space with a reproducing kernel.  The correspondence between analytic Hilbert modules of constant rank and holomorphic Hermitian bundles on domains of $\mathbb C^n$ due to Cowen and Douglas, as well as a natural analytic localization technique derived from the Hochschild cohomology of topological algebras play a major role in the proofs. A series of concrete computations, inspired by representation theory of linear groups, illustrate the abstract concepts of the paper.
\noindent
\end{abstract}

\maketitle

\section{Preliminaries and main results}\label{int}

Without aiming at completeness, the rather lengthy introduction below recalls some of the main concepts of Cowen-Douglas theory,
a localization technique in topological homology and aspects of complex Hermitian geometry as they interlace in the unitary 
classification of analytic Hilbert modules. The ideas invoked in the present work have evolved and converged from
quite distinct sources for at least half a century.

One of the basic problem in the study of a Hilbert module $\mathcal H$ over the ring of polynomials
$\C[\underline z]:=\mathbb C[z_1, \ldots , z_m]$ (or equivalently $\mathcal O(\C^m)$ module) is to find unitary
invariants (cf. \cite{dp, cg}) for $\mathcal H$.  It is not always possible to  find invariants that are
complete and yet easy to compute. There are very few instances where a set of complete invariants have been
identified.  Examples are Hilbert modules over continuous functions (spectral theory of normal operator),
contractive modules over the disc algebra (model theory for  contractive operator) and Hilbert modules in the
class $\mathrm{B}_n(\Omega)$ for a bounded domain $\Omega \subseteq \mathbb C^m$ (adjoint of multiplication
operators on reproducing kernel Hilbert spaces).  In this paper, we study  Hilbert modules consisting of holomorphic functions on some bounded domain
possessing a reproducing kernel.  Our methods apply, in particular, to submodules of Hilbert modules in $\mathrm{B}_1(\Omega)$.

\subsection{\sf  The algebraic and analytic framework}
The class $\mathrm{B}_n(\Omega)$ was introduced in \cite{cd, cd1} and an alternative approach was outlined in
\cite{cs}.  The definition of this class given below is clearly equivalent to the one in \cite[Definition 1.2]{cd} and \cite[Definition 1.1]{cs}.
\begin{defn} \label{B_n}
A Hilbert module $\mathcal H$ over the polynomial  ring  $\C[\underline z]$ is said to be in the class $\mathrm{B}_n(\Omega)$, $n\in \mathbb N$, if
\begin{enumerate}
\item[\sf{(const)}] $\dim\mathcal H/\mathfrak m_w\mathcal H = n <\infty$ for all $w\in\Omega$;
\item[\sf{(span)}] $\cap_{w\in\Omega}\mathfrak m_w\mathcal H = {0},$
\end{enumerate}
where $\mathfrak m_w$ denotes the maximal ideal in $\C[\underline z]$ at $w$.
\end{defn}
Recall that if $\mathfrak m_w\mathcal H$ has finite codimension then $\mathfrak m_w\mathcal H$ is a closed subspace of $\mathcal H$. Throughout this paper we call $\dim\mathcal H/\mathfrak m_w\mathcal H$ {\it the rank} of the analytic module at the point
$w$.
For any Hilbert module $\mathcal H$ in $\mathrm B_n(\Omega)$, the analytic localization $\mathcal O\hat\otimes_{\mathcal O(\C^m)}\mathcal H$ is a locally free module when restricted to $\Omega$, see for details \cite{EP}. Let us denote in short
$$
\hat{\mathcal H}:=\mathcal O\hat\otimes_{\mathcal O(\C^m)}\mathcal H\big{|}_{\Omega},
$$
and let $E_{\mathcal H} = \hat{\mathcal H}|_\Omega$ be the associated holomorphic vector bundle. Fix $w\in\Omega$. The minimal
projective resolution of the maximal ideal at the point $w$ is given by
the Koszul complex $K_{\centerdot}(z-w, \mathcal H)$, where $K_p(z-w, \mathcal H)= \mathcal H\otimes
\wedge^p(\C^m)$ and the connecting maps $\delta_p(w):K_p \ra K_{p-1}$ are defined, using  the standard basis vectors $e_i,\,1\leq i\leq m$ for $\C^m$, by
$$
\delta_p(w) (fe_{i_1}\wedge\ldots\wedge e_{i_p})= \d_{j=1}^p(-1)^{j-1}(z_j-w_j)\cdot f e_{i_1}\wedge\ldots\wedge
\hat e_{i_j}\wedge\ldots\wedge e_{i_p}.
$$
Here, $z_i \cdot f$ is the module multiplication.
In particular $\delta_1(w):\mathcal H\oplus\ldots\oplus\mathcal H\ra \mathcal H$ is defined by
$(f_1,\ldots,f_m)\mapsto \sum_{j=1}^m(M_j-w_j)f_j$, where $M_i$  is the operator $M_j: f \mapsto z_j \cdot f$, for
$1\leq j \leq m$ and $f \in \mathcal H$. The $0$-th homology group of the complex, $H_0(K_{\centerdot}(z-w,
\mathcal H))$ is same as $\mathcal H/\mathfrak m_w\mathcal H$. For $w\in\Omega$, the map $\delta_1(w)$ induces a map localized at w,
$$
K_1(z-w, \hat{\mathcal H}_w)\overset{{\delta_1}_w(w)}\longrightarrow K_0(z-w, \hat{\mathcal H}_w).
$$ Then $\hat{\mathcal H}_w= \mbox{coker}~{{\delta_1}_w}(w)$ is a locally free $\mathcal O_w$ module and the fiber of the associated
holomorphic vector bundle $E_{\mathcal H}$ is given by
$$
E_{\mathcal H, w}=\hat {\mathcal H}_w\otimes_{\mathcal O_w}\mathcal O_w/\mathfrak m_w\mathcal O_w.
$$
We identify $E_{\mathcal H, w}^*$ with $\ker~\delta_1(w)^*$. Thus $E^*_{\mathcal H}$ is a
Hermitian holomorphic vector bundle on $\Omega^*:=\{\bar z:z\in\Omega\}$. Let $D_{\mathbf M^*}$ be the commuting $m$-tuple $({M_1}^*,\ldots ,{M_m}^*)$ from $\mathcal H$ to $\mathcal
H\oplus\ldots\oplus\mathcal H$. Clearly $\delta_1(w)^*= D_{(\mathbf M-w)^*}$ and $\ker~\delta_1(w)^*=\ker D_{(\mathbf M-w)^*} =
\cap_{j=1}^m \ker ({M_j}-w_j)^*$ for $w \in\Omega$.

Let $\mathrm{Gr}(\mathcal H, n)$ be the rank $n$ Grassmanian on the Hilbert module $\mathcal H$.  The map
$\Gamma:\Omega^* \to \mathrm{Gr}(\mathcal H, n)$ defined by $\bar w \mapsto \ker D_{(\mathbf M-w)^*}$ is shown to be
holomorphic in \cite{cd}. The pull-back of the canonical vector bundle on $\mathrm{Gr}(\mathcal H, n)$ under
$\Gamma$ is then the holomorphic Hermitian vector bundle $E^*_\mathcal H$ on the open set $\Omega^*$. One of the
main theorems of \cite{cd} states that isomorphic Hilbert modules correspond to equivalent vector bundles and
vice-versa. Examples of these are Hardy and the Bergman modules over the ball and the poly-disc in $\mathbb
C^m$.

\subsection{\sf Submodules of Hilbert modules possessing a reproducing kernel}
Let $\mathcal H$ be a Hilbert module in $\mathrm{B}_1(\Omega)$ possessing a non-degenerate reproducing kernel $K(z,w)$, that is, $K(w,w) \not = 0,\, w\in \Omega$. We will often write  $K_w$ for the function $K(\cdot, w)$. Then
$E^*_{\mathcal H}\cong\mathcal O_{\Omega^*}$, that is, the associate holomorphic vector bundle is trivial, with $K_w$ as
a non-vanishing global section. For modules in $\mathrm B_1(\Omega)$, the curvature of the vector
bundle $E^*_{\mathcal H}$ is a complete invariant. However, in many natural examples of submodules of Hilbert
modules from the class $\mathrm B_1(\Omega)$, the dimension of the joint kernel does not remain constant.
For instance, in the case of $H_0^2(\mathbb D^2):=\{f \in H^2(\mathbb D^2): f(0) = 0\}$ (cf. \cite{D}), we have
$$
\dim \ker D_{(\mathbf M - w)^*} = \dim H_0^2(\mathbb D^2) \otimes_{\C[z_1, z_2]} \mathbb C_w  = \begin{cases} 1
&\mbox{if}\,\, w \not = (0,0)\\ 2  & \mbox{if}\,\,w = (0,0). \end{cases}
$$
Here $\mathbb C_w$ is the one dimensional module over the polynomial ring $\C[z_1,z_2]$, where the module
action is given by the map $(f, \lambda) \mapsto f(w) \lambda$ for $f\in \mathcal C_2$ and $\lambda \in \mathbb
C_w\cong \mathbb C$.

In examples like the one given above, the map $\bar w \mapsto \ker D_{(\mathbf M -w)^*} $ is not holomorphic on all
of $\mathbb D^2$ but only on $\mathbb D^2 \setminus \{(0,0)\}$. 
However, we recall that the map $w \mapsto \dim (\mathcal M / \mathfrak m_w \mathcal M)$ is upper semi-continuous and the 
jump locus, which is the set $\Omega \setminus \{w: \dim (\mathcal M / \mathfrak m_w \mathcal M) = {\rm constant}\}$, is an analytic set.  
In this paper, we begin a systematic study of a class of 
submodules of kernel Hilbert modules (over the polynomial ring $\C[\underline z]$) in $\mathrm{B}_1(\Omega)$ characterized by requiring that 
$\dim (\mathcal M / \mathfrak m_w \mathcal M)$, $w\in \Omega$, is finite.  
\begin{defn}
A Hilbert module $\mathcal M$ over the polynomial  ring  $\C[\underline z]$ is said to be in the class
$\mathfrak{B}_1(\Omega)$ if
\begin{enumerate}
\item[\sf{(rk)}] possess a reproducing kernel $K$ (we don't rule out the possibility: $K(w,w)=0$ for  $w$ in
some closed subset $X$ of $\Omega$) and \item[\sf{(fin)}] The dimension of $\mathcal M/\mathfrak m_w\mathcal M$
is finite for all $w\in \Omega$.
\end{enumerate}
\end{defn}

The following Lemma isolates a large class of elements  from  $\mathfrak B_1(\Omega)$ which belong to
$\mathrm{B}_1(\Omega_0)$ for some open subset $\Omega_0\subseteq \Omega$.
\begin{lem} \label{B_1}
Suppose $\mathcal M \in \mathfrak B_1(\Omega)$ is the closure of a polynomial ideal $\mathcal I$.  Then
$\mathcal M$ is in $\mathrm{B}_1(\Omega)$ if the ideal  $\mathcal I$ is singly generated while if it is
generated by the polynomials $p_1, p_2,\ldots , p_t$, then $\mathcal M$ is in $\mathrm{B}_1(\Omega\setminus X)$
for $X=\cap_{i=1}^t \{z:p_i(z)=0\}\cap \Omega$.
\end{lem}
\begin{proof}
The proof is a refinement of the argument given in \cite[pp. 285]{dm2}. Let $\gamma_w$ be any eigenvector at $w$
for the adjoint of the module multiplication, that is, $M_p^* \gamma_w = \overline{p(w)} \gamma_w$ for $p \in
\C[\underline z]$.

First, assume that the module $\mathcal M$ is generated by the single polynomial, say $p$.  In this case,
$K(z,w) = p(z) \chi(z,w) \overline{p(w)}$ for some positive definite kernel $\chi$ on all of $\Omega$. Set
$K_1(z,w) = p(z)\chi(z,w)$ and note that  $K_1(\cdot,w)$ is a non-zero eigenvector at $w \in \Omega$.  We have
\begin{eqnarray*}
\langle pq, \gamma_w\rangle = \langle p, M_q^*\gamma_w\rangle = \langle p, \ov{q(w)}\gamma_w\rangle &=&
q(w)\langle p, \gamma_w\rangle\\ & = &\frac{\langle pq, K(\cdot, w)\rangle\langle p, \gamma_w\rangle}{p(w)} =
\langle pq, \ov{\langle p, \gamma_w\rangle}K_1(\cdot, w)\rangle.
\end{eqnarray*}
Since vectors of the form $\{pq: q\in\C[\underline z]\}$ are dense in $\mathcal M$, it follows that $\gamma_w
=\ov{\langle p, \gamma_w\rangle} K_1(\cdot, w)$ and  the proof is complete in this case.

Now, assume that $p_1, \ldots , p_t$ is a set of generators for the ideal $\mathcal I$.  Then for $w \not \in
X$, there exist a $k\in\{1,\ldots,t\}$ such that $p_k(w)\neq 0$. We note that for any $i,\, 1\leq m$,
$$
p_k(w)\langle p_i, \gamma_w\rangle = \langle p_i, M_{p_k}^*\gamma_w\rangle = \langle p_i p_k , \gamma_w \rangle
= \langle p_k, M_{p_i}^*\gamma_w\rangle = p_i(w)\langle p_k, \gamma_w\rangle.
$$
Therefore we have
\begin{eqnarray*}
\langle \sum_{i=1}^tp_iq_i, \gamma_w\rangle = \sum_{i=1}^t\langle p_iq_i, \gamma_w\rangle = \sum_{i=1}^t\langle
p_i, M_{q_i}^*\gamma_w\rangle &=& \sum_{i=1}^tq_i(w)\langle p_i, \gamma_w\rangle\\ &=& \sum_{i=1}^t\langle
p_iq_i, \frac{\ov{\langle p_k, \gamma_w\rangle}K(\cdot, w)}{\ov{p_k(w)}}\rangle.
\end{eqnarray*}
Let $c(w) = \frac{\langle p_k, \gamma_w \rangle}{p_k(w)}$. Hence
$$
\sum_{i=1}^t \langle p_i q_i, \gamma_w \rangle =  \langle \sum_{i=1}^tp_iq_i, \ov{c(w)}K(\cdot, w)\rangle .
$$
Since vectors of the form $\{\sum_{i=1}^tp_iq_i: q_i\in\C[\underline z],\, 1\leq i\leq t\}$ are dense in
$\mathcal M$, it follows that $\gamma_w=\ov{c(w)}K(\cdot, w)$ completing the proof of the second half.
\end{proof}
\subsection{\sf The sheaf construction}
From the work of \cite{cd, cd1}, it is known that invariants for holomorphic Hermitian bundles are not easy
to compute. We show how to do this for a large family of examples.  It then becomes clear that to find easily
computable invariants, we must look elsewhere.  Using techniques from commutative algebra and complex analytic
geometry, in the framework of Hilbert modules, we have obtained some new invariants.

Let us consider a Hilbert module $\mathcal M$ in the class $\mathfrak B_1(\Omega)$ which is a submodule of some Hilbert module $\mathcal H$ in
$\mathrm B_1(\Omega)$, possessing a nondegenerate reproducing kernel $K$. Clearly then we have the following module map
\begin{eqnarray}\label{mot}
\mathcal O\hat\otimes_{\mathcal O(\C^m)}\mathcal M\longrightarrow\mathcal O\hat\otimes_{\mathcal O(\C^m)}\mathcal H\cong \mathcal O_{\Omega}.
\end{eqnarray}
Let $\mathcal S^{\mathcal M}$ denotes the range of the composition map in the above equation. Then the stalk of
$\mathcal S^{\mathcal M}$ at $w\in\Omega$ is given by $\{(f_1)_w \mathcal O_w + \cdots + (f_n)_w \mathcal
O_w : f_1, \ldots , f_n \in \mathcal M \}$

Motivated by the above construction and the analogy with the correspondence of a vector bundle with a locally free sheaf \cite{Wells},
we construct a sheaf $\mathcal S^\mathcal M$ for the Hilbert module $\mathcal M$ over the polynomial ring $\C[\underline z]$,
in the class $\mathfrak B_1(\Omega)$.
The sheaf $\mathcal S^{\mathcal M}$ is the subsheaf of the sheaf of holomorphic functions $\mathcal O_\Omega$ whose stalk $\mathcal S^{\mathcal M}_w$ at $w \in
\Omega$ is
$$ \big \{(f_1)_w \mathcal O_w + \cdots + (f_n)_w \mathcal
O_w : f_1, \ldots , f_n \in \mathcal M \big \}, $$ or equivalently,
$$\mathcal S^{\mathcal M}(U) = \Big\{\sum_{i=1}^n \big ( {f_i}_{| U}
\big) g_i: f_i \in \mathcal M, g_i \in \mathcal O(U) \Big \}$$ for $U$ open in $\Omega$.

Following the proof of \cite[Theorem 2.3.3]{cg}, which is a consequence of the well known Cartan Theorems $A$ and $B$,
it is not hard to see that if $\mathcal M$ is any module in
$\mathfrak B_1(\Omega)$ with a finite set of generators $\{f_1, \ldots , f_t\}$, then for any $f\in \mathcal
S^\mathcal M$ we have
\begin{equation} \label{genmod}
f = f_1 g_1 + \cdots + f_t g_t
\end{equation}
for some $g_1, \ldots , g_t \in \mathcal O(\Omega)$.  Consequently, if $\mathcal M_1$ and $\mathcal M_2$ are
isomorphic modules in $\mathfrak B_1(\Omega)$ which are finitely generated, then $\mathcal S^{\mathcal
M_1}$ and  $\mathcal S^{\mathcal M_2}$ are isomorphic as modules. This isomorphism is
implemented by extending the map given on the generators of $\mathcal M_1$ to the module $\mathcal S^{\mathcal
M_1}$. It is easily seen to be well-defined using \eqref{genmod}.

It is clear that if the Hilbert module $\mathcal M$ is in the class ${\mathrm B}_1(\Omega)$, then the sheaf
$\mathcal S^\mathcal M$ is locally free. Also, if the Hilbert module is taken to be the   maximal set of
functions vanishing on an analytic hyper-surface $\mathcal Z$, then  the sheaf $S^{\mathcal M}$
coincides with the ideal sheaf $\mathcal I_\mathcal Z(\Omega)$  and therefore it is coherent (cf.\cite{GR}).
However, much more is true
\begin{prop}
For any Hilbert module $\mathcal M$ in $\mathfrak B_1(\Omega)$, the sheaf $\mathcal S^\mathcal M$ is
coherent.
\end{prop}
\begin{proof}
The sheaf $S^{\mathcal M}$ is generated by the family $\{f:f\in\mathcal M\}$ of global sections of the
sheaf $\mathcal O(\Omega)$. Let $J$ be a finite subset of $\mathcal M$ and $S^{\mathcal M}_J \subseteq \mathcal
O(\Omega)$ be the subsheaf generated by the sections $f,~f\in J$. It follows (see \cite[Corollary 9, page.
130]{gr}) that $S^{\mathcal M}_J$ is coherent. The family $\{S^{\mathcal M}_J : J \:\mbox{is a finite subset
of}\: \mathcal M\}$   is increasingly filtered, that is, for any two finite subset $I$ and $J$ of $\mathcal M$,
the union $I \cup J$ is again a finite subset of $\mathcal M$ and $S^{\mathcal M}_I\cup S^{\mathcal M}_J\subset
S^{\mathcal M}_{I\cup J}$. Also, clearly $S^{\mathcal M}=\bigcup_JS^{\mathcal M}_J$. Using Noether's lemma
\cite[page. 111]{GR} which says that every increasingly filtered family must be stationary, we conclude that the
sheaf $S^{\mathcal M}$ is coherent.
\end{proof}

For $w\in \Omega$, the coherence of $\mathcal S^\mathcal M$ ensures the existence of $m,n \in \mathbb N$ and an open neighborhood $U$ of $w$ such that
$$
(\mathcal O^m)_{|U} \to (\mathcal O^n)_{|U} \to (\mathcal S^\mathcal M)_{|U} \to 0
$$
is an exact sequence.  Thus $$\Big \{ \Big (\mathcal S^\mathcal M_w /\mathfrak m_w \mathcal S^\mathcal M_w \Big )^* : w \in \Omega \Big \}$$
defines a holomorphic linear space on $\Omega$ (cf. \cite[1.8 (p. 54)]{F}).  Although, we have not used this  correspondence in any essential manner, we expect it to be a useful tool in the investigation of some of the questions we raise here.

The coherence of the sheaf $\mathcal S^\mathcal M$ implies, in particular, that the stalk $(\mathcal
S^\mathcal M)_w$ at $w\in \Omega$ is generated by a finite number of elements $g_1,  \ldots ,g_d$ from $\mathcal
O(\Omega)$. If $K$ is the reproducing kernel for $\mathcal M$ and $w_0\in \Omega$ is a fixed but arbitrary
point, then for $w$ in a small neighborhood $\Omega_0$ of $w_0$, we obtain the following decomposition theorem.
\begin{thm} \label{decomp}
Suppose $g^0_i,\,1\leq i \leq d,$ be a minimal set of generators for the stalk $\mathcal S^\mathcal M_{w_0}$.
Then
\begin{enumerate}
\item[(i)]there exists a open neighborhood $\Omega_0$ of $w_0$ such that
$$K(\cdot, w):=K_w = g^0_1(w)K^{(1)}_w + \cdots + g^0_n(w)K^{(d)}_w,\, w \in \Omega_0$$ for some
choice of anti-holomorphic functions $K^{(1)}, \ldots , K^{(d)}:\Omega_0 \to \mathcal M$, \item[(ii)] the
vectors $K^{(i)}_{w},\, 1\leq i \leq d$, are linearly independent in $\mathcal M$ for $w$ in some small
neighborhood of $w_0$, \item[(iii)]the vectors $\{K^{(i)}_{w_0}\mid  1\leq i \leq d\}$ are uniquely determined
by these generators $g^0_1, \ldots , g_d^0$, \item[(iv)] the linear span of the set of vectors
$\{K^{(i)}_{w_0}\mid  1\leq i \leq d\}$ in $\mathcal M$ is independent of the generators $g^0_1, \ldots ,
g_d^0$, and \item[(v)] $M_p^*K^{(i)}_{w_0} = \ov{p(w_0)} K^{(i)}_{w_0}$ for all $i,\, 1\leq i \leq d$, where $M_p$ denotes the
module multiplication by the polynomial $p$.
\end{enumerate}
\end{thm}
For simplicity, we have stated the decomposition theorem for Hilbert modules 
consisting of holomorphic functions taking values in $\mathbb C$.  However, all the tools that we use for the proof work equally well in the case of vector valued holomorphic functions.  Consequently, it is not hard to see that the  theorem remains valid in this more general set-up.

\subsection{\sf Gleason's property and privilege}
It is evident from the above theorem that the dimension of the joint kernel of the adjoint of the multiplication operator $D_{\mathbf M^*}$ at a point $w_0$ is
greater or equal to the number of minimal generators of the stalk $\mathcal S^\mathcal M_{w_0}$ at $w_0\in
\Omega$, that is,
\begin{eqnarray}
\label{ineq} \dim\mathcal M/(\mathfrak m_{w_0} \mathcal M) \geq  \dim\mathcal S^\mathcal M_{w_0}/\mathfrak
m_{w_0} \mathcal S^\mathcal M_{w_0} .
\end{eqnarray}


It would be interesting to produce a Hilbert module $\mathcal M$ for which the  inequality of \eqref{ineq} is strict.
Leaving aside this question, for the moment, we go on to identify several classes of Hilbert modules for which
we have equality in \eqref{ineq}.

A  Hilbert module $\mathcal M$ over the polynomial ring  $\C[\underline z]$ is said to be an {\em analytic
Hilbert module} (cf. \cite{cg}) if we assume that
\begin{enumerate}
\item[\sf (rk)] it consists of holomorphic functions on a bounded domain $\Omega \subseteq \mathbb C^m$ and
possesses a reproducing kernel $K$, \item[\sf (dense)] the polynomial ring $\C[\underline z]$ is dense in it,
\item[\sf (vp)] the set of virtual points $\{w \in \mathbb C^m : p \mapsto p(w), \, p \in
\C[\underline z],\:\mbox{~extends continuously to }\mathcal M\}$, is $\Omega$.
\end{enumerate}
We apply Lemma \ref{B_1} to analytic Hilbert modules, which are singly generated by the constant function $1$,
to conclude that they must be in $\mathrm{B}_1(\Omega)$.   Evidently, in this case, we have equality in
\eqref{ineq}. However,  we have equality in many more cases.  For example,  suppose $\mathcal I$ is a polynomial
ideal and $[\mathcal I]$ is the closure of $\mathcal I$ in some analytic Hilbert module  $\mathcal M$.  Then for
$[\mathcal I]$, we have equality in \eqref{ineq} as well.

Let us again consider a Hilbert module $\mathcal M$ in the class $\mathfrak B_1(\Omega)$ which is a submodule of some Hilbert module $\mathcal H$ in
$\mathrm B_1(\Omega)$, possessing a nondegenerate reproducing kernel $K$. We note that the module map 
$$
\mathcal O\hat\otimes_{\mathcal O(\C^m)}\mathcal M\longrightarrow \mathcal S^{\mathcal M}
$$
induced from \eqref{mot} is surjective. This naturally defines a map 
$$
\mathcal M/\mathfrak m_{w_0}\mathcal M\cong \mathcal O_{w_0}/\mathfrak m_{w_0}\mathcal O_{w_0}\otimes \mathcal M
\longrightarrow \mathcal S^\mathcal M_{w_0}/\mathfrak
m_{w_0} \mathcal S^\mathcal M_{w_0} 
$$
for $w \in \Omega$. 
The map given above can be constructed similarly for any Hilbert module $\mathcal M\in\mathfrak B_1(\Omega)$. The question of
equality in \eqref{ineq} is same as the question of whether this map is an isomorphism and can be interpreted as a
global factorization problem. To be more specific, we say that the module $\mathcal M \in\mathfrak B_1(\Omega)$ possesses
{\it Gleason's property at a point $w_0 \in \Omega$} if for every element $f \in  \mathcal M$ vanishing at $w_0$ there
are $f_1,...,f_m \in \mathcal M$ such that $f = \d_{i=1}^m(z_i - w_{0i})f_i$.

\begin{prop}\label{glea}  The Hilbert module $\mathcal M$ has Gleason's property at $w_0$ if and only if
$$
\dim \mathcal M/\mathfrak m_{w_0}\mathcal M = \dim
\mathcal S^\mathcal M_{w_0}/\mathfrak
m_{w_0} \mathcal S^\mathcal M_{w_0}.
$$
\end{prop}
We note the following corollary.
\begin{cor}\label{glea1}
For an analytic Hilbert module and its submodules which arises as closure of an ideal in the polynomial ring $\C[\underline z]$,
Gleason's problem is solvable.
\end{cor}


It is well known that Gleason's probelm is solvable in the space of all analytic functions, that is,  assuming that the domain
$\Omega$ is pseudo-convex, it follows that  for any $f \in \mathcal M$ with $f(w_0) = 0$, we have
$$
f = \d_{i=1}^m(z_i - w_{0i})f_i,\,f_i\in \mathcal O(\Omega).
$$
 see for instance \cite[Theorem 7.2.9]{krantz} or \cite{EP}.
Thus Gleason's problem asks that the functions $f_i$ can be chosen from the Hilbert
module $\mathcal M$. 

This is a special case of a more general division problem for Hilbert modules. 
To fix ideas, we consider the following setting:  let $\mathcal M$ be an analytic Hilbert module with
the domain $\Omega$ disjoint of its essential spectrum, let $A\in M_{p,q}(\O(\overline{\Omega}))$ be a matrix of analytic functions
defined in a neighborhood of $\overline{\Omega}$, where $p,q$ are positive integers, and let $f \in \mathcal
M^p$. GIven a solution  $u\in\O(\Omega)^q$ to the linear
equation $ A u = f $, is it true that $u \in {\mathcal M}^q$? Numerous ``hard analysis" questions, such as problems of moduli, or Corona Problem, can be put into this framework.

We study below this very division problem in conjunction with an earlier work of the third author
\cite{PS} dealing with the ``disc" algebra $\mathcal A(\Omega)$ instead of Hilbert modules, and within the general
concept of ``privilege'' introduced by Douady more than forty years ago \cite{D0,D1}.

Below we only focus on the case of Bergman space. Specifically, the $\A(\Omega)$-module $\mathcal N = {\rm coker} ( A : \mathcal M\otimes_\mathbb C \mathbb C^p \longrightarrow
\mathcal M\otimes_\mathbb C \mathbb C^q)$ is called {\it privileged with respect to the module $\mathcal M$} if
it is a Hilbert module in the quotient metric and there exists a resolution
\begin{equation}\label{Bergmanresolution}
0\rightarrow \mathcal M\otimes_\mathbb C \mathbb C^{n_p}\xrightarrow{d_p}\cdots\rightarrow \mathcal
M\otimes_\mathbb C \mathbb C^{n_1}\xrightarrow{d_1} \mathcal M\otimes_\mathbb C \mathbb C^{n_0}\rightarrow
\mathcal N\rightarrow 0,
\end{equation}
where $d_q\in M_{n_{q+1},n_q}(\A(\Omega))$. Note that implicitly in the statement is assumed that the range of
the operator $A$ is closed at the level of the Hilbert module $\mathcal M$.

An affirmative answer to the division problem is equivalent to the question of ``privilege'' in case of the
Bergman module on a strictly convex bounded domain $\Omega$  with smooth boundary.
\begin{thm}\label{Privilege} Let $\Omega \subset \C^m$ be a
strictly convex domain with smooth boundary, let $p,q$ be positive integers and let $A \in M_{p,q}(\A(\Omega))$
be a matrix of analytic functions belonging to the disk algebra of $\Omega$. The following assertions are
equivalent:
\begin{enumerate}
\item[(a)] The analytic module ${\rm coker} (A : L^2_a(\Omega)^p \longrightarrow L^2_a(\Omega)^q)$ is privileged
with respect to the Bergman space; \item[(b)] The function $ \zeta \mapsto {\rm rank}\  A(\zeta), \ \  \zeta \in
\partial\Omega,$ is constant; \item[(c)] Let $f \in L^2_a(\Omega)^q$. The equation $ Au= f$ has a solution $u
\in L^2_a(\Omega)^p$ if and only if it has a solution $u \in \O(\Omega)^p$.
\end{enumerate}
\end{thm}

While we have stated our results for the Bergman module,  they remain true for the Hardy space $H^2(\partial
\Omega)$, that is, the closure of entire functions in the $L^2$-space with respect to the surface area measure
supported on $\partial \Omega$. Also, the results remain true for the Bergman or Hardy spaces of a poly-domain
$\Omega = \Omega_1 \times\cdots\times \Omega_d$, where $\Omega_j \subset \C, \ 1 \leq j \leq d,$ are convex
bounded domains with smooth boundary in $\mathbb C$. For these Hilbert modules, the notion of the sheaf model
from the earlier work of \cite{pu1,pu2} coincides with the sheaf model described here.

\subsection{\sf The Hermitian structure}
It follows, from the Lemma \ref{B_1} that $H^2_0(\mathbb D^2)$ is in $\mathrm{B}_1(\mathbb D^2 \setminus \{(0,0)\})$.
Thus the machinery of \cite{cd, cs} applies here.  But explicit calculation of unitary invariants are somewhat difficult.  As was
pointed out in \cite{dp}, the dimension of the localization $H_0^2(\mathbb D^2) \otimes_{\mathbb C[z_1,z_2]}
\mathbb C_w$, $w\in \mathbb D^2$ is an invariant of the module $H_0^2(\mathbb D^2)$. Therefore, it may not be
desirable to exclude the point $(0,0)$ altogether in any attempt to study the module $H_0^2(\mathbb D^2)$.
Fortunately, implicit in the proof of Theorem 2.2 in \cite{cs}, there is a  construction which makes it possible
to write down invariants on all of $\mathbb D^2$. This theorem assumes only that the module multiplication
has closed range as in Definition \ref{B_n}.
Therefore, it plays a significant role in the study of the class of Hilbert modules $\mathfrak B_1(\Omega)$.

We also note, from Theorem \ref{decomp}, that the map $\Gamma_K: \Omega_0^* \to \mathrm{Gr}(\mathcal M,d)$ defined
by $\Gamma_K(\bar{w}) = (K^{(1)}_{w}, \ldots ,
K^{(d)}_{w})$ is holomorphic. The pull-back of the canonical bundle on ${\mathrm Gr}(\mathcal M, d)$ under
$\Gamma_K$ then defines a holomorphic Hermitian vector bundle on the open set $\Omega_0^*$.  Unfortunately, the
decomposition of the reproducing kernel given in Theorem \ref{decomp} is not canonical except when the stalk is
singly generated.  In this special case, the holomorphic Hermitian bundle obtained in this manner is indeed
canonical.  However, in general, it is not clear if this vector bundle contains any useful information.
Suppose we have equality in \eqref{ineq} for a Hilbert module $\mathcal M$. Then it is possible to obtain a
canonical decomposition following  \cite{cs}, which leads in the same manner as above, to the construction of a
Hermitian holomorphic vector bundle in a neighborhood of each point $w\in \Omega$.

For any fixed but arbitrary $w_0\in \Omega$ and a small enough neighborhood $\Omega_0$  of $w_0$, the proof of
Theorem 2.2 from \cite{cs} shows  the existence of a holomorphic function $P_{w_0}: \Omega_0 \to \mathcal
L(\mathcal M)$ with the property that the operator $P_{w_0}$ restricted to the subspace $\ker D_{(\mathbf M -
w_0)^*}$ is invertible. The range of $P_{w_0}$ can then be seen to be equal to  the kernel of the operator
$\mathbb P_0 D_{(\mathbf M - w)^*}$, where $\mathbb P_0$ is the orthogonal projection onto ${{\mathrm ran}
D_{(\mathbf M - w_0)^*}}$.
\begin{lem}\label{const}
The dimension of $\ker \mathbb P_0 D_{(\mathbf M - w)^*}$ is constant in a suitably small neighborhood of $w_0
\in \Omega$, say $\Omega_0$.
\end{lem}

Let $\{e_0, \ldots , e_k\}$ be a basis for $\ker (\mathbf M - w_0)^*$. Since $P_{w_0}$ is holomorphic on
$\Omega_0$, it follows that  $\gamma_1(w):= P_{w_0}(w) e_1 , \ldots , \gamma_k(w) :=P_{w_0}(w) e_k$ are
holomorphic on $\Omega_0$.  Thus  $\Gamma: \Omega_0 \to {\mathrm Gr}(\mathcal M, k)$, given by $\Gamma(w) = \ker
\mathbb P_0 D_{(\mathbf M - w)^*}$, defines a holomorphic Hermitian vector bundle $\mathcal P_0$ on $\Omega_0$
of rank $k$ corresponding to the Hilbert module $\mathcal M$.

\begin{thm}\label{csb}   If any two Hilbert modules $\mathcal M$ and $\mathcal {\tilde M}$ from $\mathfrak B_1(\Omega)$ are isomorphic via an unitary module map, then the corresponding holomorphic Hermitian vector bundles
$\mathcal P_0$ and $\mathcal {\tilde P}_0$ on $\Omega_0^*$ are equivalent.
\end{thm}



\subsection{\sf Organization}
We now describe the  organization of the paper. In Section \ref{sh}, we prove the decomposition theorem \ref{decomp}.
The equivalence of the Gleason property with the equality in \eqref{ineq} is shown in Section \ref{anHM}. At the end of this Section, we give a simple proof of  a conjecture from
\cite{dmv} for smooth points. This was first proved in  \cite{dg}.  In Section \ref{putinar},
we prove that the Bergman module is privileged.  In Section \ref{co.sh}, we show that the sheaf model of
\cite{pu1, pu2} coincides with the one proposed here if the Hilbert modules are assumed to be privileged.
Finally in Section \ref{inv}, we construct a  Hermitian holomorphic vector bundle following \cite{cs}.  We show
how to extract invariants for the Hilbert module from this vector bundle. These invariants are not easy to
compute in general, but we provide explicit computations for a class of examples in Section \ref{lm} and  in
Section \ref{nk}, we give detailed calculations of an invariant which is somewhat easier to compute.

\subsection{\rm }{\sf Index of notations}\vskip .75em

\begin{tabular}{ll}
$\C[\underline z]$ & the polynomial ring $\C[z_1,\ldots,z_m]$ of $m$- complex variables  \\

$\mathfrak m_w$ & maximal ideal of $\C[\underline z]$ at the point $w\in\C^m$\\

$\Omega$ & a bounded domain in $\C^m$ \\

  $\Omega^*$ & $\{\bar z: z\in\Omega\}$ \\

  $\D^m$ &  the unit polydisc in $\C^m$ \\

$M_i$ & module multiplication by the co-ordinate function $z_i,\, 1\leq i\leq m$\\

  $M_i^*$ & adjoint of $M_i$ \\

  $D_{(\mathbf M - w)^*}$ & the operator $\mathcal M\ra \mathcal M\oplus\ldots\oplus\mathcal M$ defined by $f\mapsto ((M_j - w_j)^*f)_{j=1}^m $\\

  $\hat {\mathcal H}$ & the analytic localization $\mathcal O\hat\otimes_{\mathcal O(\C^m)}\mathcal H$ of the Hilbert module $\mathcal H$\\

  $ B_n(\Omega)$ & Cowen-Douglas  class of operators, $n\geq 1$\\

  $\partial^{\alpha},\bar\partial^{\alpha}$ & $\partial^{\alpha}=\frac{\partial^{|\alpha|}}{\partial z_1^{\alpha_1}\cdots z_m^{\alpha_m}},\bar{\partial}^{\alpha}=\frac{\partial^{|\alpha|}}{\partial\bar{z}_1^{\alpha_1}\cdots \bar{z}_m^{\alpha_m}}$ for $\alpha = (\alpha_1,\ldots,\alpha_m)\in

  {\mathbb Z^+\times\cdots\times\mathbb Z^+}$ and \\ {} & $|\alpha|=\sum_{i=1}^m\alpha_i$\\

$q(D)$ & the differential operator $\sum_{\alpha}a_{\alpha}\partial^{\alpha}$ corresponding to the polynomial $q = \sum_{\alpha}a_{\alpha}z^{\alpha}$ \\

$\mathcal S^{\mathcal M}$ & the analytic submodule of $\mathcal O_\Omega$, corresponding to the Hilbert module $\mathcal M \in \mathfrak{B}(\Omega)$ \\

  $K(z,w)$ & a reproducing kernel\\

  $E(w) $ & the evaluation functional (the linear functional induced by $K(\cdot, w)$)\\

  $\parallel\cdot\parallel_{\bar \Delta(0;r)}$ & supremum norm\\

  $\parallel\cdot\parallel_2$ & $L^2$ norm with respect to the volume measure\\

$\mathbb V_w(\mathcal F)$ & the characteristic space at $w$, which is $\{q\in\C[\underline z]: q(D)f\big{|}_w=0 \mbox{~for ~all ~}f\in\mathcal F\}$\\

   & for some set  $\mathcal F$ of holomorphic functions\\

 $[\mathcal I]$ & the closure of the polynomial ideal $\mathcal I \subseteq \mathcal M$ in some Hilbert module $\mathcal M$ \\

 $\mathcal A(\Omega)$ & the "disk" algebra $\O(\Omega) \cap

  C(\overline{\Omega})$ over $\Omega$ \\

  $\mathcal O(\ov \Omega)$ & the space of germs

  of analytic functions defined in a neighborhood of $\ov \Omega$\\






  $\mathbb P_0$ & orthogonal projection onto ${{\mathrm ran}~

D_{(\mathbf M - w_0)^*}}$\\

  $\mathcal P_w$ & $\ker \mathbb P_0 D_{(\mathbf M - w)^*}$ for $w\in\Omega$ \\

\end{tabular}

\section{The proof of the decomposition theorem}\label{sh}

\begin{proof}[Proof of Theorem \ref{decomp}]
For simplicity of notation, we assume, without loss of generality, that $0=w_0\in\Omega$. Let
$\{e_n\}_{n=0}^{\infty}$ be a orthonormal basis for $\mathcal M$. Let $\{e_n^*\}_{n=0}^{\infty}$ be the
corresponding dual basis of ${\mathcal M}^*$, that is, $e_n^*(e_j)=\delta_{nj}$, Kronecker delta, $n,j\in\mathbb
N\cup\{0\}$. Let $E(z)$ be the evaluation functional at the point $z\in\Omega$. Clearly $E(z)\in{\mathcal M}^*$,
as $\mathcal M$ posses a reproducing kernel $K$. So $E(z)=\textstyle \sum_{n=0}^{\infty}a_n(z)e_n^*$. Now,
$$
e_n(z)=E(z)e_n=\{\d_{k=0}^{\infty} a_k(z)e_k^*\}e_n=\d_{k=0}^{\infty}a_k(z)e_k^*(e_n)=
\d_{k=0}^{\infty}a_k(z)\delta_{kn}=a_n(z).
$$

It follows from \cite[Theorem 2, page. 82]{gr} that for every element $f$ in $\mathcal S^{\mathcal M}_0$, and
therefore in particular for every $e_n$, we have
$$
e_n(z)=\d_{i=1}^d g^0_i(z)h_i^{(n)}(z),\,z\in\Delta(0;r)
$$
for some holomorphic functions $h_i^{(n)}$ defined on the closed polydisk $\bar\Delta(0;r) \subseteq \Omega$.
Furthermore, they can be chosen with the bound  ${\parallel h_i^{(n)}\parallel}_{\infty, \bar\Delta(0;r)}\leq C
{\parallel e_n\parallel}_{\infty, \bar\Delta(0;r)}$ for some positive constant $C$ independent of $n$. Although, the decomposition is not necessarily with respect to the standard coordinate system at $0$, we will be using only a point wise estimate. Consequently, in the equation given above, we have chosen not to emphasize the change of variable involved and we have,
$$
E(z)= \sum_{n=0}^{\infty}\{\d_{i=1}^d g^0_i(z)h_i^{(n)}(z)\}e_n^*=\d_{i=1}^d
g^0_i(z)\{\d_{n=0}^{\infty}h_i^{(n)}(z)e_n^*\}.
$$
Setting  $H_i(z)$ to be the sum $\textstyle\sum_{n=0}^{\infty}h_i^{(n)}(z)e_n^*$, we can write
$E(z)=\textstyle\sum_{i=1}^d H_i(z)g^0_i(z)$ on $\Delta(0;r)$. For the proof of part (i), we need to show that
$H_i(z)\in{\mathcal M}^*$ where $z\in\Delta(0;r)$. Or, equivalently, we have to show that  $\sum_{n=0}^{\infty}{
|{h_i^{(n)}}(z)|}^2<{\infty}$ for each $z \in \Delta(0;r)$. First, using the estimate on $h_i^{(n)}$, we have
$$
|{h_i^{(n)}}(z)|\leq{\parallel h_i^{(n)}\parallel}_{\bar\Delta(0;r)}\leq C {\parallel
e_n\parallel}_{\bar\Delta(0;r)}.
$$
We prove below the  inequality $\sum_{n=0}^{\infty}{\parallel e_n\parallel}_{\infty
\bar\Delta(0;r)}^2<{\infty},$ completing the proof of part (i). We prove, more generally, that for
$f\in{\mathcal M}$, \begin{equation}\label{ie1}{\parallel f\parallel}_{\bar\Delta (0;r)} \leq C^\prime{\parallel
f\parallel}_{2,{\bar\Delta (0;r)}},\end{equation} where ${\parallel .\parallel}_2$ denotes the $L^2$ norm with
respect to the volume measure on $\bar\Delta (0;r)$
It is evident from the proof that the  constant  $C^\prime$ may be chosen to be independent of the functions
$f$.

Any function $f$  holomorphic on $\Omega$ belongs to the Bergman space $L^2_a(\Delta (0;r+\varepsilon))$ as long
as $\Delta (0;r+\varepsilon) \subseteq \Omega$.  We can surely pick $\varepsilon > 0$ small enough to ensure
$\Delta (0;r+\varepsilon) \subseteq \Omega$.   Let $B$ be the Bergman kernel of the Bergman space $L^2_a(\Delta
(0;r+\varepsilon))$.
Thus we have
$$\mid f(w)\mid ~=~ \mid\langle f, B(\cdot,w)\rangle\mid~\leq~ {\parallel f\parallel}_{2,\Delta (0;r+\varepsilon)}{B(w,w)}^{\frac{1}{2}},\, w\in\Delta (0;r+\varepsilon).$$
Since the function $B(w,w)$ is bounded on compact subsets of $\Delta (0;r+\varepsilon)$, it follows that
${C^\prime}^2:=\sup \{B(w,w) : w\in\bar\Delta (0;r)\}$ is finite. We therefore see that
$${\parallel f\parallel}_{\bar\Delta (0;r)}=\sup \{\mid f(w)\mid : w\in\bar\Delta (0;r)\}\leq C^\prime {\parallel f\parallel}_{2,\Delta (0;r+\varepsilon)}.$$  Since $\varepsilon >0$ can be chosen arbitrarily close to $0$, we infer the inequality \eqref{ie1}.

The inequality \eqref{ie1} implies, in particular, that
$$
\d_{n=0}^{\infty} {\parallel e_n\parallel}_{\bar\Delta (0;r)}^2\leq
C^\prime\d_{n=0}^{\infty}\displaystyle\int_{\bar\Delta (0;r)}{\mid e_n(z)\mid }^2 dz_1\wedge d{\bar z}_1\wedge
\cdots \wedge dz_m\wedge d{\bar z}_m .
$$
We have shown that the evaluation functional $E(z) \in \mathcal M^*$ is of the form $\sum_{n=0}^\infty e_n(z)
e_n^*$ and hence the function $G(z):=\textstyle\sum_{n=0}^{\infty}|e_n(z)|^2$ is finite for each $z\in \Omega$.
The sequence of positive continuous functions  $G_k(z):=\sum_{n=0}^{k}{|e_n(z)| }^2$ converges uniformly to $G$
on $\bar\Delta (0;r)$.  To see this, we note   that
\begin{eqnarray*}
{\parallel  G_k -  G \parallel^2}_{\bar\Delta(0;r)} &\leq & {C^\prime}^2 \displaystyle\int_{\bar\Delta(0;r)}{|
G_k(z)-G(z)|}^2 dz_1\wedge d{\bar z}_1\wedge \cdots \wedge dz_m\wedge d{\bar z}_m\\ &\leq & {C^\prime}^2
\displaystyle\int_{\bar\Delta(0;r)}\{\d_{n=k+1}^{\infty}{\mid e_n(z)\mid }^2\}^2 dz_1\wedge d{\bar z}_1\wedge
\cdots \wedge dz_m\wedge d{\bar z}_m,
\end{eqnarray*}
which tends to $0$ as $k\ra\infty$. So, by monotone convergence theorem, we can interchange the integral and the
infinite sum to conclude
$$\d_{n=0}^{\infty} {\parallel e_n\parallel}_{\bar\Delta (0;r)}^2\leq C\displaystyle\int_{\bar\Delta
(0;r)}\d_{n=0}^{\infty}{\mid e_n(z)\mid }^2 dz_1\wedge d{\bar z}_1\wedge \cdots \wedge dz_m\wedge d{\bar z}_m <
{\infty}$$ as $G$ is a continuous function on $\bar\Delta(0;r)$.  This shows that
$$\d_{n=0}^{\infty}{\mid {h_i^{(n)}(z)}\mid}^2 \leq K\d_{n=0}^{\infty} {\parallel e_n\parallel}_{\bar\Delta
(0;r)}^2<\infty.$$ Hence $H_i(z)\in{\mathcal M}^*$.

Now, for each $i$, $1\leq i \leq d$, there exist $K_w^{(i)}\in\mathcal M$ such that $H_i(w)f=\langle f,
K_w^{(i)}\rangle$. Let us set $K_i(z,w):=K_w^{(i)}(z)$.  The function $K_i$ is holomorphic in the first variable
and antiholomorphic in the second by definition. For $w \in \Delta (0;r)$, we have $\langle f,
K(\cdot,w)\rangle=E(w)f=\sum_{i=1}^d g^0_i(w)H_i(w)f=\sum_{i=1}^d\langle f, \bar g^0_i(w)K_i(\cdot,
w)\rangle=\langle f, \sum_{i=1}^d\bar g^0_i(w)K_i(\cdot, w)\rangle$ for all $f\in\mathcal M$. Hence $K(z,
w)=\sum_{i=1}^d\bar g^0_i(w)K_i(z, w) $ and (i) is proved.

To prove the statement in (ii), at $0$, we have to show that whenever there exist complex numbers $\alpha_1,
\ldots , \alpha_d$ such that  $\sum_{i=1}^d{\alpha}_iK_i(z, 0)=0$, then ${\alpha}_i=0$ for all $i$.  We assume,
on the contrary, that there exists some $i\in{1,\ldots ,d}$ such that $ {\alpha}_i\neq 0$. Without loss of
generality, we assume ${\alpha}_1\neq 0$, then $K_1(z, 0)=\sum_{i=2}^d{\beta}_iK_i(z, 0)$ where ${\beta}_i=\frac
{{\alpha}_i}{{\alpha}_1}, 2\leq i\leq d $. This shows that $K_1(z, w)-\sum_{i=2}^d{\beta}_iK_i(z, w)$ has a zero
at $w=0$. From \cite[Theorem 7.2.9]{krantz}, it follows that
$$
K_1(z, w)-\d_{i=2}^d{\beta}_iK_i(z, w)=\d_{j=1}^m\bar w_jG_j(z, w)
$$
for some function $G_j:\Omega\times \Delta(0;r)\ra\mathbb C, 1\leq j\leq m$, which is holomorphic in the first
and antiholomorphic in the second variable. So, we can write
\begin{eqnarray*}
K(z, w) &=& \d_{i=1}^d\bar g^0_i(w)K_i(z, w)=\bar g^0_1(w)K_1(z, w)+\d_{i=2}^d\bar g^0_i(w)K_i(z, w)\\ &=& \bar
g_1^0(w)\{\d_{i=2}^d{\beta}_iK_i(z, w)+\d_{j=1}^m\bar w_jG_j(z, w)\}+\d_{i=2}^d\bar g^0_i(w)K_i(z, w)\\ &=&
\d_{i=2}^d(\bar g^0_i(w)+{\beta}_i\bar g^0_1(w))K_i(z, w)+\d_{j=1}^m\bar w_j\bar g^0_1(w)G_j(z, w).
\end{eqnarray*}
For $f\in\mathcal M$ and $w\in \Delta(0;r)$, we have
$$
f(w) = \langle f, K(\cdot, w)\rangle = \d_{i=2}^d( g^0_i(w)+\bar{\beta}_i g^0_1(w))\langle f,K_i(z, w) \rangle +
g^0_1(w)\langle f,\d_{j=1}^m\bar w_jG_j(z, w)\rangle.
$$
We note that $\langle f,\sum_{j=1}^m\bar w_jG_j(z, w)\rangle$ is a holomorphic function in $w$ which vanishes at
$w=0$ It then follows that $\langle f,\sum_{j=1}^m\bar w_jG_j(z, w)=\sum_{j=1}^m w_j\tilde{G}_j(w)$ for some
holomorphic functions  $\tilde G_j,\,1\leq j\leq m$ on $\Delta(0;r)$. Therefore, we have
$$
f(w) =\d_{i=2}^d( g^0_i(w)+\bar{\beta}_i g^0_1(w))\langle f,K_i(z, w) \rangle + g^0_1(w)\d_{j=1}^m
w_jg^0_1(w)\tilde{G}_j(w).
$$
Since the sheaf $\mathcal S^{\mathcal M}\big{|}_{\Delta(0;r)}$ is generated by the Hilbert module $\mathcal M$, it
follows that  the set $\{g^0_2+\bar{\beta}_2g^0_1,\ldots ,g^0_d+\bar{\beta}_dg^0_1,z_1g^0_1,\ldots ,z_mg^0_1\}$
also generates $\mathcal S^{\mathcal M}\big{|}_{\Delta(0;r)}$. In particular, they generate the stalk at $0$. This, we
claim, is a contradiction. Suppose $A\subset \mathcal S^{\mathcal M}_{0}$ is generated by germs of the functions
$g^0_2+\bar\beta_2g^0_1,\ldots ,g^0_d+\bar\beta_dg^0_1$. Let ${\mathfrak m}(\mathcal O_0)$ denotes the the only
maximal ideal of the local ring $\mathcal O_0$, consisting of the germs of functions vanishing at $0$. Then  it
follows that $${\mathfrak m}(\mathcal O_0)\{\mathcal S^{\mathcal M}_{0}/A\}=\mathcal S^{\mathcal M}_{0}/A.$$
Using Nakayama's lemma (cf. \cite[p.57]{tay}), we see that $\mathcal S^{\mathcal M}_{0}/A=0$, that is, $\mathcal
S^{\mathcal M}_{0}=A$. This contradicts the minimality of the generators of the stalk at $0$ completing the
proof of first half of (ii).

To prove the slightly stronger statement, namely, the independence of the vectors  $K^{(i)}_{w_0}$ , $1\leq i
\leq d$ in a small neighborhood of $0$, consider the Grammian $\big (\!\big (\langle K^{(i)}_w, K^{(j)}_w
\rangle \big )\! \big )_{i,j=1}^d$.  The determinant of this Grammian is nonzero at $0$.  Therefore it remains
non-zero in a suitably small neighborhood of $0$ since it is a real analytic function on $\Omega_0$.
Consequently, the vectors $K^{(i)}_w,~i=1,\ldots ,d$ are linearly independent for all $w$ in this neighborhood.

To prove the statement in (iii), that is, to prove that $K^{(i)}_{0}$ are uniquely determined by the generators
$g^0_i$ ,  $1\leq i\leq d$.  We will let  $g^0_i$ denote the germ of $g^0_i$ at $0$ as well. The uniqueness of
the set of vectors $K^{(i)}_{0}$ is clearly the same as the uniqueness of the set of linear functionals $H_i$,
$1 \leq i \leq d$.  Thus enough to show if $\sum_{i=1}^d g^0_i(z)\{H_i(z)-\tilde H_i (z)\}=0$ for some choice of
$\tilde H_i$, $1 \leq i \leq d$, then $(H_i-\tilde H_i)(0) = 0$.  But then we have
$\sum_{n=0}^{\infty}\sum_{i=1}^d g^0_i(z)\{h_i^{n}(z)-\tilde h_i^n (z)\}e_n^*=0$. Hence, for each n
$$
\d_{i=1}^d g^0_i(z)\{h_i^{n}(z)-\tilde h_i^n (z)\}=0.
$$
Fix $n$ and let $\alpha_i(z)=h_i^{n}(z)-\tilde h_i^n (z)$. In this notation,   $\sum_{i=1}^d
g^0_i(z)\alpha_i(z)=0$. Now we claim that $\alpha_i(0)=0$ for all $i\in\{1,\ldots ,d\}$. If not, we may assume
$\alpha_1(0)\neq 0$. Then the germ of $\alpha_1$ at $0$ is a unit in $_m{\mathcal O}$. Hence we can write, in
$_m\mathcal O$, $$g_1^0= -(\d_{i=2}^d g_i^0{\alpha}_{i0}){{\alpha}_{10}}^{-1},$$ where ${\alpha}_{i0}$ denotes
the germs of the analytic functions ${\alpha}_{i}$ at $0$,  $1 \leq i \leq d$. This is a contradiction, as
$g_1^0,\ldots ,g^0_d$ is a minimal set of generators of the stalk $\mathcal S^\mathcal M_0$ by hypothesis. As a
result,  $h_i^{n}(0)=\tilde h_i^n (0)$ for all $i\in\{1,\ldots ,d\}$ and $n\in{{\mathbb N}\cup \{0\}}$. Then
$H_i(0)=\tilde H_i (0)$ for all $i\in\{1,\ldots ,d\}$.  This completes the proof of (iii).

To prove the statement in (iv), let $\{g^0_1, \ldots , g^0_d\}$ and $\{\tilde g^0_1, \ldots ,\tilde g^0_d\}$ be
two sets of generators for $\mathcal S^\mathcal M_{0}$ both of which are minimal.  Let $K^{(i)}$ and $\tilde
K^{(i)}$, $1\leq i\leq d$, be the corresponding vectors  that appear in the decomposition of the reproducing
kernel $K$ as in (i). It is enough to show  that $$\mbox{span}_{\C}\{K_i(z, 0): 1\leq i\leq
d\}=\mbox{span}_{\C}\{\tilde K_i(z, 0): 1\leq i\leq d\}.$$ There exists holomorphic functions $\phi_{ij}$,
$1\leq i,j \leq d$, in a small enough neighborhood of $0$ such that $\tilde g^0_i= \sum_{j=1}^d\phi_{ij}g^0_j$.
It now follows that
\begin{eqnarray*}K(z,w) &=& \d_{i=1}^d\bar{\tilde g}^0_i(w)\tilde K_i(z,w)~=~\d_{i=1}^d(\d_{j=1}^d\bar\phi_{ij}(w)\bar g^0_j(w))
\tilde K_i(z,w)\\ &=& \d_{j=1}^d\bar g^0_j(w)(\sum_{i=1}^d\bar\phi_{ij}(w)\tilde K_i(z,w)) \end{eqnarray*} for
$w$ possibly from an even smaller neighborhood of $0$.  But $K(z,w) = \sum_{j=1}^d\bar g^0_j(w) K_j(z,w)$ and
uniqueness at the point $0$  implies that
$$K_j(z,0)=\d_{i=1}^d\bar\phi_{ij}(0)\tilde K_i(z,0)$$ for $1\leq j\leq d$. So,  we have $\mbox{span}_{\C}
\{K_i(z, 0): 1\leq i\leq d\}\subseteq\mbox{span}_{\C}\{\tilde K_i(z, 0): 1\leq i\leq d\}.$  Writing $g^0_j$ in
terms of $\tilde g^0_i$, we get the other inclusion.

Finally, to prove the statement in (v), let us apply ${M_j}^*$ to both sides of the decomposition of the
reproducing kernel $K$ given in part (i) to obtain  $\bar w_jK(z, w)=\sum_{i=1}^d\bar g^0_i(w){M_j}^*K_i(z, w)$.
Substituting $K$ from the first equation, we get $\sum_{i=1}^d\bar g^0_i(w)(M_j-w_j)^*K_i(z, w)=0$. Let
$F_{ij}(z, w)=(M_j-w_j)^*K_i(z, w)$. For a fixed but arbitrary $z_0\in\Omega$, consider the equation
$\sum_{i=1}^d\bar g^0_i(w)F_{ij}(z_0, w)=0$. Suppose there exists $k,1\leq k\leq d$ such that $F_{kj}(z_0,
0)\neq 0$. Then
$$
g^0_{k}=\{\overline {F_{kj}(z_0,\cdot)}_0\}^{-1}\d_{i=1,i\neq k}^dg^0_{i}\overline {F_{ij}(z_0,\cdot)}_0.
$$
This is a contradiction. Therefore $F_{ij}(z_0, 0)=0$, $1\leq i\leq d$, and for all $z_0\in\Omega$. So
${M_j}^*K_i(z, 0)=0$, $1\leq i\leq d,\, 1\leq j\leq m $. This completes the proof of the theorem.\end{proof}


\begin{rem} Let $\mathcal I$ be an ideal in the polynomial ring $\C[\underline z]$.
Suppose $\mathcal M\supset \mathcal I$ and that $\mathcal I$ is dense in
$\mathcal M$.  Let $\{p_i\in \C[\underline z]: 1\leq i\leq t\}$ be a minimal set of generators for the ideal $\mathcal I$. Let $V(\mathcal I)$ be the zero variety of the ideal $\mathcal I$. If $w\notin V(\mathcal I)$, then $\mathcal S^{\mathcal M}_{w}= {}_m{\mathcal O}_{w}$. Although $p_1,\ldots ,p_t$ generate the stalk at every point, they are not necessarily a minimal set of generators. 
For example, let $\mathcal I = < z_1(1+z_1), z_1(1-z_2), z_2^2>\subset \C[z_1, z_2]$. The functions $z_1(1+z_1),
z_1(1-z_2), z_2^2$ form a minimal set of generators for the ideal $\mathcal I$.  Since $1+ z_1$ and $1-z_2$ are
units in ${}_2\mathcal O_0$,  it follows that the functions $z_1$ and $z_2^2$ form a minimal set of generators
for the stalk  $\mathcal S^\mathcal M_0$.
\end{rem}

\section{The joint kernel at $w_0$ and the stalk $\mathcal S^\mathcal M_{w_0}$} \label{anHM}
Let $g^0_1,\ldots,g^0_d$ be a minimal set of generators for the stalk $\mathcal S^\mathcal M_{w_0}$ as before.
For $f\in\mathcal S^\mathcal M_{w_0}$,  we can find holomorphic functions $f_i,\,1\leq i \leq d$ on some small
open neighborhood $U$ of $w_0$ such that $f=\sum_{i=1}^dg^0_if_i$ on $U$.
We write
$$
f=\d_{i=1}^dg^0_if_i=\d_{i=1}^dg^0_i\{f_i-f_i(w_0)\}+\d_{i=1}^dg^0_if_i(w_0).
$$
on $U$.  Let ${\mathfrak m}(\mathcal O_{w_0})$ be the maximal ideal (consisting of the germs of holomorphic
functions vanishing at the point $w_0$) in the local ring $\mathcal O_{w_0}$ and
${\mathfrak m}(\mathcal O_{w_0})\mathcal S^{\mathcal M}_{w_0}={\mathfrak m_{w_0}\mathcal S^{\mathcal M}_{w_0}}$.
Thus the linear span of the
equivalence classes $[g^0_{1}],\ldots ,[g^0_{d}]$ is the quotient module ${\mathcal S^{\mathcal
M}_{w_0}}/{\mathfrak m_{w_0}\mathcal S^{\mathcal M}_{w_0}} $.  Therefore we have
$$
\mbox{~dim~}{\mathcal S^{\mathcal M}_{w_0}}/{\mathfrak m_{w_0}\mathcal S^{\mathcal M}_{w_0}}
 \leq d.
$$
It turns out that the elements $[g^0_{1}],\ldots ,[g^0_{d}]$ in the quotient module are linearly independent.
Then $\mbox{~dim~}{\mathcal S^{\mathcal M}_{w_0}}/{\mathfrak m_{w_0}\mathcal S^{\mathcal M}_{w_0}} =
d$. To prove the linear independence, let us consider the equation $\sum_{i=1}^d\alpha_i[g^0_{i}]=0$ for some
complex numbers $\alpha_i,\, 1\leq i\leq d$, or equivalently,  $\sum_{i=1}^d\alpha_ig^0_{i} \in \mathfrak
m(\mathcal O_{w})S^{\mathcal M}_{w}$. Thus there exists holomorphic functions $f_i,\, 1\leq i\leq d$,  on a
small enough neighborhood of $w_0$ and vanishing at $w_0$ such that $\sum_{i=1}^d(\alpha_i-f_i)g_i=0$.  Now
suppose $\alpha_k\neq 0$ for some $k, 1\leq k\leq d$. Then we can write $g^0_{k}=-\sum_{i\neq k}
 (\alpha_k-f_k)_0^{-1}(\alpha_i-f_i)_0g^0_{i}$ which is a contradiction.
From  the decomposition theorem \ref{decomp}, it follows that
\begin{equation}\label{steq}
\dim\displaystyle\cap_{j=1}^m \ker ({M_j}-w_{0j})^*\geq\sharp \{minimal ~generators ~for ~S^{\mathcal M}_{w_0}\}
=\dim{\mathcal S^{\mathcal M}_{w_0}}/{\mathfrak m_{w_0}\mathcal S^{\mathcal M}_{w_0}}.
\end{equation}

We will impose natural conditions on the Hilbert module $\mathcal M$, always assumed to be in the class
$\mathfrak B_1(\Omega)$, so as to ensure equality in \eqref{ineq} which is also the inequality given above. Let
$V(\mathcal M):=\{w\in\Omega : f(w)=0, \mbox{~for~all~} f\in\mathcal M\}$ . Then for $w_0 \not\in V(\mathcal
M)$, the number of minimal generators for the stalk at $w_0$ is one, in fact, $S^{\mathcal M}_{w_0}=
{}_m{\mathcal O}_{w_0}$. Also, $\mbox{dim~} \ker D_{(\mathbf{M}-w_0)^*}=1$ since the joint kernel at $w_0$  is
spanned by the Kernel function $K(\cdot, w_0)$ of $\mathcal M$ for $w_0\not\in V(\mathcal M)$. Therefore,
outside the zero set, we have  equality in \eqref{ineq}.  For a large class of Hilbert modules, we will show,
even on the zero set, that the reverse inequality is valid.  For instance, for Hilbert modules of rank $1$ over
$\C[\underline z]$, we have equality everywhere. This is easy to see:
$$
1\geq \dim \mathcal M\otimes_{{\mathcal C}_m}\C_{w_0} = \dim\displaystyle\cap_{j=1}^m \ker ({M_j}-w_{0j})^*\geq
\dim{\mathcal S^{\mathcal M}_{w_0}}/{\mathfrak m_{w_0}\mathcal S^{\mathcal M}_{w_0}}\geq 1.
$$

To understand the more general case, consider the map $i_w:\mathcal M\lo{\mathcal M}_{w}$ defined by $f\mapsto
f_{w}$, where $f_{w}$ is the germ of the function $f$ at $w$. Clearly, this map is a vector space isomorphism
onto its image. The linear space ${\mathcal M}^{(w)}:=\sum_{j=1}^m(z_j-{w}_j) \mathcal M=\mathfrak m_w\mathcal
M$ is closed since $\mathcal M$ is assumed to be  in $\mathfrak B_1(\Omega)$. Then the map $f\mapsto f_{w}$
restricted to ${\mathcal M}^{(w)}$ is a linear isomorphism from ${\mathcal M}^{(w)}$ to $({\mathcal
M}^{(w)})_{w}$.
Consider
$$\mathcal M \overset{i_w}\lo \mathcal S^{\mathcal M}_{w}\overset{\pi}\lo {\mathcal S^{\mathcal M}_{w}}/{\mathfrak
m(\mathcal O_{w})\mathcal S^{\mathcal M}_{w}},$$ where 
${\pi}$ is the quotient map. Now we have a map $\psi:{\mathcal M_{w}}/ {(\mathcal M^{(w)})_{w}}\lo{\mathcal
S^{\mathcal M}_{w}}/\{{\mathfrak m(\mathcal O_{w})\mathcal S^{\mathcal M}_{w}}\}$ which is well defined because
$(\mathcal M^{(w)})_{w}\subseteq\mathcal M_{w}\cap\mathfrak m(\mathcal O_{w})\mathcal S^{\mathcal M}_{w}$.
Whenever $\psi$ can be shown to be one-one, equality in \eqref{ineq} is forced. To see this, note that
${\mathcal M}\ominus{\mathcal M^{(w)}}\cong\mathcal M/{\mathcal M}^{(w)}$ and
$$\cap_{j=1}^m \ker ({M_j}-w_j)^*=\cap_{j=1}^m{\{\mathrm{ran} (M_j-{w}_j)\}}^{\perp}=\mathcal M
\ominus \d_{j=1}^m(z_j-{w}_j)\mathcal M={\mathcal M}\ominus{\mathcal M^{(w)}}.$$ Hence
\begin{eqnarray}\label{pivot}d\leq \dim \cap_{j=1}^m \ker ({M_j}-w_j)^*=\dim{\mathcal M}/
{\mathcal M^{(w)}}\leq\dim {\mathcal S^{\mathcal M}_{w}}/{\mathfrak m(\mathcal O_{w})\mathcal S^{\mathcal
M}_{w}} =d.\end{eqnarray} Suppose $\psi(f)=0$ for some $f \in \mathcal M$. Then $f_w\in {\mathfrak m(\mathcal
O_{w})\mathcal S^{\mathcal M}_{w}}$ and consequently, $f= \d_{i=1}^{m}(z_i-w_i)f_i$ for holomorphic functions
$f_i$, $1\leq i \leq m$, on some small open set $U$. The main question is if the functions $f_i$, $1\leq i \leq
m$, can be chosen from the Hilbert module $\mathcal M$.  We isolate below, a class of Hilbert modules for which
this question has an affirmative answer.

Let $\mathcal H$ be a Hilbert module over the polynomial ring $\C[\underline z]$ in the class
$\mathrm{B}_1(\Omega)$.  Pick, for each $w\in \Omega$, a $\mathbb C$ - linear subspace $\mathbb V_w$  of the
polynomial ring $\C[\underline z]$ with the property that it is invariant under the action of the  partial
differential operators $\{\frac{\partial}{\partial z_1},...,\frac{\partial}{\partial z_m}\}$ (see \cite{cg}).
Set
$$\mathcal M(w) = \{f\in\mathcal H: q(D)f|_{w}=0 \mbox{~for~all~}q\in\mathbb V_{w}\}.$$
For $f\in\mathcal M(w)$ and $q\in\mathbb V_{w}$, $$q(D)(z_jf)|_{w} = w_jq(D)f|_{w}+
 \frac{\partial q}{\partial z_j}(D)f|_{w}=0.$$
Now, the assumption on $\mathbb V_{w}$ ensure that $\mathcal M(w)$ is a module. We consider below, the class of
(non-trivial) Hilbert modules which are of the form  $\mathcal M:=\bigcap_{w\in \Omega} \mathcal M(w)$.
It is  easy to see that $$w\notin V(\mathcal M)\mbox{~if~and~only~if} ~\mathbb V_w=\{0\} \mbox{~if~ and~
only~if} ~\mathcal M(w)=\mathcal H.$$ Therefore,  $\mathcal M =\bigcap_{w\in V(\mathcal M)}\mathcal M(w)$. Let
$\mathbb V_w(\mathcal M) := \{q\in\C[\underline z]: q(D)f\big{|}_w=0\mbox{~for ~ all~} f\in\mathcal
 M\}$.  We note that $\mathbb V_w(\mathcal M)=\mathbb V_w$.
Fix a point in $V(\mathcal M)$, say $w_0$. Consider $\tilde{\mathbb V}_{w_0}(\mathcal M)=\{q\in\C[\underline z]:
\frac{\partial q}{\partial z_i}\in \mathbb V_{w_0}(\mathcal M),\: 1\leq i\leq m$\}. For $w\in V(\mathcal M)$,
let
\begin{eqnarray*}\label{chsp}
\mathbb V_w^{w_0}(\mathcal M) =
\begin{cases}
\mathbb V_w(\mathcal M) &\: \mbox{if}\: w\neq w_0\\
\tilde{\mathbb V}_{w_0}(\mathcal M)&\:\mbox{if}\: w=w_0.
\end{cases}
\end{eqnarray*}
Now, define $\mathcal M^{w_0}(w)$ to be the submodule (of $\mathcal H$)  corresponding to the family of the
$\mathbb C$-linear subspaces $\mathbb V_w^{w_0}(\mathcal M)$ and let $\mathcal M^{w_0}=\bigcap_{w\in V(\mathcal
M)}\mathcal M^{w_0}(w)$. So we have $\mathbb V_w(\mathcal M^{w_0})=\mathbb V_w^{w_0}(\mathcal M)$. For
$f\in\mathcal M^{(w_0)}$, we have $f=\d_{j=1}^m(z_j-w_{0j})f_j$, for some choice of $f_1, \ldots ,f_m\in\mathcal
M$. Now for any $q\in\C[\underline z]$, following \cite{cg}, we have
\begin{eqnarray}\label{cg1}q(D)f=\d_{i=1}^mq(D)\{(z_j-w_{0j})f_j\}=\d_{i=1}^m\{(z_j-w_{0j})q(D)f_j+\frac{\partial
q} {\partial z_j}(D)f_j\}.\end{eqnarray} For $w \in V(\mathcal M)$ and $f \in \mathcal M^{(w_0)}$, it follows
from the definitions that
$$
q(D)f\big |_w =
\begin{cases}
\d_{i=1}^m\{(w_j-w_{0j})q(D)f_j|_w+\frac{\partial q} {\partial
z_j}(D)f_j|_w\} = 0 &\: q\in\mathbb V_w^{w_0},\: w\neq w_0\\
\d_{i=1}^m\{\frac{\partial q} {\partial z_j}(D)f_j|_{w_0}\}=0&\:
 q\in\mathbb V^{w_0}_{w_0},\:w=w_0.
\end{cases}
$$
Thus $f\in\mathcal M ^{(w_0)}$ implies that $f\in\mathcal M^{w_0}(w)$ for each $w\in V(\mathcal M)$.  Hence
$\mathcal M^{(w_0)} \subseteq \mathcal M^{w_0}$. Now we describe the Gleason property for $\mathcal M$ at a
point $w_0$.

\begin{defn}
We say that $\mathcal M:= \cap_{w\in V(\mathcal M)}\mathcal M(w)$ has the  Gleason property at a point $w_0\in
V(\mathcal M)$ if $\mathcal M^{w_0}=\mathcal M^{(w_0)}$.
\end{defn}

Analogous to the definition of $\mathbb V_{w_0}(\mathcal M)$ for a Hilbert module $\mathcal M$, we define the
space $\mathbb V_{w_0}(\mathcal S^{\mathcal M}_{w_0}) = \{q\in\C[\underline z]:q(D)f\big |_{w_0}=0,\: f_{w_0}\in
\mathcal S^{\mathcal M}_{w_0}\}$. It will be useful to record the relation between $\mathbb V_{w_0}(\mathcal M)$
and $\mathbb V_{w_0}(\mathcal S^{\mathcal M}_{w_0})$ in a separate lemma.
\begin{lem}\label{obs}For any Hilbert module in $\mathfrak B_1(\Omega)$ and $w_0\in\Omega$, we have $\mathbb V_{w_0}(\mathcal M)=\mathbb V_{w_0}(\mathcal S^{\mathcal M}_{w_0})$.
\end{lem}
\begin{proof} We note that the inclusion $\mathbb V_{w_0}(\mathcal S^{\mathcal
M}_{w_0})\subseteq \mathbb V_{w_0}(\mathcal M)$ follows from $\mathcal M_{w_0}\subseteq \mathcal S^{\mathcal
M}_{w_0}$. To prove the reverse inclusion, we need to show that $q(D)h|_{w_0}=0$ for $h\in\mathcal S^{\mathcal
M}_{w_0}$, for all $q\in\mathbb V_{w_0}(\mathcal M)$. Since $h\in\mathcal S^{\mathcal M}_{w_0}$, we can find
functions $f_1, \ldots , f_n \in \mathcal M$ and $g_1, \ldots , g_n \in \mathcal O_{w_0}$ such that
$h=\d_{i=1}^nf_ig_i$ in some small open neighborhood of $w_0$.
Therefore, it is enough to show that  $q(D)(fg)|_{w_0}=0$ for $f\in\mathcal M$,  $g$ holomorphic in a
neighborhood, say $U_{w_0}$ of $w_0$, and $q \in \mathbb V_{w_0}(\mathcal M)$. We can choose $U_{w_0}$ to be a
small  enough polydisk such that $g=\d_{\alpha}a_{\alpha}(z-w_0)^\alpha, ~z\in U_{w_0}$.  We then see that
$q(D)(fg)=\d_{\alpha}a_{\alpha}q(D)\{(z-w_0)^\alpha f\}$ for $z\in U_{w_0}$. 
Clearly, $(z-w_0)^\alpha f$ belongs to $\mathcal M$ whenever $f\in \mathcal M$. Hence $q(D)\{(z-w_0)^\alpha
f\}|_{w_0}=0$ and we have $q(D)(fg)|_{w_0}=0$ completing the proof of $\mathbb V_{w_0}(\mathcal M) \subseteq
\mathbb V_{w_0}(\mathcal S^{\mathcal M}_{w_0})$..
\end{proof}

We will show that we have equality in \eqref{ineq} for all Hilbert modules with the Gleason property.

\begin{proof}[Proof of proposition \ref{glea}]  We first show that $\ker(\pi\circ i_{w_0} ) = \mathcal M^{w_0}$. Showing $\ker(\pi\circ i_{w_0} )\subseteq\mathcal M^{w_0}$ is same as showing  $\mathcal M_{w_0}\cap\mathfrak m_{w_0}\mathcal S^{\mathcal M}_{w_0}\subseteq(\mathcal
M^{w_0})_{w_0}$. 
We claim that \begin{equation}\label{gp}\mathbb V_{w_0}(\mathfrak m_{w_0}\mathcal S^{\mathcal M}_{w_0})=\mathbb V_{w_0}^{w_0}(\mathcal M).\end{equation} If $f\in\mathfrak m_{w_0}\mathcal S^{\mathcal M}_{w_0}$, then there exists
$f_{j}\in \mathcal S_{w_0}^{\mathcal M}$ such that  $f=\d_{i=1}^m(z_j-w_{0j})f_{j}$. From equation (\ref{cg1}),
we have
$$q\in\mathbb V_{w_0}(\mathfrak m_{w_0}\mathcal S^{\mathcal M}_{w_0}) \mbox{~if~ and~ only~ if~}
\frac{\partial q}{\partial z_j}\in\mathbb V_{w_0}(\mathcal S^{\mathcal M}_{w_0})=\mathbb V_{w_0}(\mathcal M)
\mbox{~for~ all~} j, 1\leq j\leq m.$$ Now, from lemma \ref{obs}, we see that $\frac{\partial q}{\partial
z_j}\in\mathbb V_{w_0}(\mathcal M)$  $ 1\leq j\leq m$, if and only if $q\in\mathbb V_{w_0}^{w_0}(\mathcal M)$,
which proves our claim.
So for $f\in\mathcal M$, if  $f_{w_0}\in\mathfrak m_{w_0}\mathcal S^{\mathcal M}_{w_0}$, then
$f\in\mathcal M^{w_0}(w)$ for all $w\in V(\mathcal M)$.  Hence $f\in\mathcal M^{w_0}$ and as a result, we have $
\mathcal M_{w_0}\cap\mathfrak m_{w_0}\mathcal S^{\mathcal M}_{w_0}\subseteq(\mathcal
M^{w_0})_{w_0}$.

Now let $f\in\mathcal M^{w_0}$. From \eqref{gp} it follows that $$
f\in\{g\in\mathcal O_{w_0}: q(D)g\big{|}_{w_0}=0 \mbox{~for~all~} q\in\mathbb V_{w_0}(\mathfrak m_{w_0}\mathcal S^{\mathcal M}_{w_0})\}.
$$
Then from \cite[Prposotion 2.3.1]{cg} we have $f\in\mathfrak m_{w_0}\mathcal S^{\mathcal M}_{w_0}$. Therefore $f\in\ker (\pi\circ i_{w_0}$ and $\ker(\pi\circ i_{w_0} ) = \mathcal M^{w_0}$.

Next we show that the map $\pi\circ i_{w_0}$ is onto. Let $\sum_{i=1}^nf_ig_i\in\mathcal S^{\mathcal M}_{w_0}$, where $f_i\in\mathcal M$ and $g_i$'s are holomorphic function in some neighborhood of $w_0, \, 1\leq i\leq n$. We need to show that there exist $f\in\mathcal M$ such that the class $[f]$ is equal to $[\sum_{i=1}^nf_ig_i]$ in $\mathcal S^{\mathcal M}_{w_0}/\mathfrak m_{w_0}\mathcal S^{\mathcal M}_{w_0}$. Let us take $f=\sum_{i=1}^nf_ig_i(w_0)$. Then 
$$
\sum_{i=1}^nf_ig_i - f = \sum_{i=1}^nf_i\{g_i -g_i(w_0)\}\in\mathfrak m_{w_0}\mathcal S^{\mathcal M}_{w_0}.
$$ This completes the proof of surjectivity.

Suppose Gleason property holds for $\mathcal M$ at $w_0$. Since $\mathcal M^{(w_0)}\subseteq \ker(\pi\circ i_{w_0})$, and we have just shown that $ \ker(\pi\circ
i) = \mathcal M^{w_0}$, it follows from the Gleason property at $w_0$ that we have the equality $\ker
(\pi\circ i_{w_0})=\mathcal M^{(w_0)}$. We recall then that the map $\psi:{\mathcal M}/{\mathcal
M^{(w_0)}}\lo{\mathcal S^{\mathcal M}_{w_0}}/\{{\mathfrak m_{w_0}\mathcal S^{\mathcal M}_{w_0}}\}$
is one to one. The equality in \eqref{ineq} is established using the equations \eqref{steq} and \eqref{pivot}.

Now suppose equality holds in \eqref{ineq}. From the above, it is clear that $\mathcal M/\mathcal M^{w_0}$ is isomorphic to $\mathcal S^{\mathcal M}_{w_0}/\mathfrak m_{w_0}\mathcal S^{\mathcal M}_{w_0}$. Thus 
$$
\dim \mathcal M/\mathcal M^{w_0} = \dim \mathcal M/\mathcal M^{(w_0)}.
$$ 
But as $\mathcal M^{(w_0)}\subseteq \mathcal M^{w_0}$, we have $\mathcal M^{(w_0)} = \mathcal M^{w_0}$ and hence Gleason property holds for $\mathcal M$ at $w_0$.
\end{proof}

A class of examples of Hilbert spaces satisfying Gleason property can be found in \cite{fang}. It was shown in
\cite{fang} that Gleason property holds for analytic Hilbert module. However it is not entirely clear if it
continues to hold for submodules of analytic Hilbert module. We will show here, never the less, we have equality
in \eqref{ineq}. Let $\mathcal M$ be a submodule of an analytic Hilbert module over $\C[\underline z]$. Assume
that $\mathcal M$  is a closure of an ideal $\mathcal I \subseteq \C[\underline z]$. From \cite{cg, dg}, we note
that
$$
\dim\displaystyle\cap_{j=1}^m \ker (M_j-w_{0j})^*=\dim \mathcal I/\mathfrak m_{w_0}\mathcal I.
$$
Therefore from \eqref{steq} we have
$$
\dim \mathcal I/\mathfrak m_{w_0}\mathcal I\geq \dim{\mathcal S^{\mathcal M}_{w_0}}/{\mathfrak m_{w_0}\mathcal S^{\mathcal M}_{w_0}}.
$$
So we need to prove the reverse inequality. Fix a point $w_0\in\Omega$. Consider the  map
$$\mathcal I\overset{i_{w_0}}\longrightarrow \mathcal S^{\mathcal M}_{w_0}\overset{\pi}\longrightarrow{\mathcal S^{\mathcal
M}_{w_0}}/{\mathfrak m_{w_0}\mathcal S^{\mathcal M}_{w_0}}.$$
We will show that $\ker (\pi\circ i_{w_0})=\mathfrak m_{w_0}\mathcal I$. Let $V(\mathcal I)$ denote the zero set
of the ideal $\mathcal I$ and $\mathbb V_w(\mathcal I)$ be its characteristic space at $w$. We begin by proving
that the characteristic space of the ideal coincides with that of corresponding Hilbert module.

\begin{lem} Assume that $\mathcal M=[\mathcal I]$. Then $\mathbb V_{w_0}(\mathcal I)=\mathbb V_{w_0}(\mathcal M)$ for $w_0\in\Omega$.
\end{lem}
\begin{proof} Clearly $\mathbb V_{w_0}(\mathcal I)\supseteq\mathbb V_{w_0}(\mathcal M)$, so we prove $\mathbb V_{w_0}(\mathcal I)\subseteq\mathbb V_{w_0}(\mathcal M)$.
For $q\in\mathbb V_{w_0}(\mathcal I)$ and $f\in\mathcal M$, we show that $q(D)f|_{w_0}=0$.  Now, for each
$f\in\mathcal M$, there exists a sequence of polynomial $p_n\in\mathcal I$ such that $p_n \to f$ in the Hilbert
space norm. Recall that if $K$ is the reproducing kernel for $\mathcal M$, then
\begin{eqnarray}\label{drk}  ({\partial}^{\alpha}f)(w) ~=~ {\langle
f,\bar{\partial}^{\alpha}K(\cdot,w)\rangle}, \mbox{~for~}  \alpha\in\mathbb Z_m^+, ~w\in\Omega,
 ~f\in\mathcal M\end{eqnarray}
For $w\in\Omega$ and  a compact neighborhood $C$ of $w$, we have
\begin{eqnarray*}&&|q(D)p_n(w)-q(D)f(w)| \\ &=& |\langle p_n - f, q(\bar D)K(\cdot,w)\rangle|\leq
\parallel p_n-f\parallel_{\mathcal M}\parallel q(\bar D)K(\cdot,w)\parallel_{\mathcal M}\\ &\leq & \parallel
p_n-f\parallel_{\mathcal M}\displaystyle\sup_{w\in C}\parallel q(\bar D)K(\cdot,w)\parallel_{\mathcal
M}.\end{eqnarray*} So, in particular,  $q(D)p_n\big{|}_{w_0}\lo q(D)f\big{|}_{w_0}$ as $n\lo\infty$.  Since
$q(D)p_n\big{|}_{w_0}=0$ for all $n$, it follows that $q(D)f\big{|}_{w_0}=0$. Hence $q\in\mathbb
V_{w_0}(\mathcal M)$ and we are done.
\end{proof}

Now let $\mathcal J=\mathfrak m_{w_0}\mathcal I$. Recall \cite[Proposition 2.3]{dg} that $V(\mathcal J)\setminus
V(\mathcal I):=\{w\in\C^m:\mathbb V_w(\mathcal I)\subsetneq\mathbb V_w(\mathcal J)\}=\{w_0\}$. Here we will
explicitly write down the characteristic space.
\begin{lem}
For $w\in\C^m$, $\mathbb V_w(\mathcal J)=\mathbb V^{w_0}_w(\mathcal I)$. Here $\mathbb V^{w_0}_w(\mathcal I) =\left\{%
\begin{array}{ll}
    \mathbb V_w(\mathcal I), & \hbox{$w\neq w_0$;} \\
    \tilde{\mathbb V}_{w_0}(\mathcal I), & \hbox{$w= w_0$}\\
\end{array}%
\right . , $ and $\tilde{\mathbb V}_{w_0}(\mathcal I)=\{q\in\C[\underline z]: \frac{\partial q}{\partial z_i}\in
\mathbb V_{w_0}(\mathcal I),\: 1\leq i\leq m\}$.

\end{lem}
\begin{proof}
Since $\mathcal J\subset\mathcal I$, we have $\mathbb V_w(\mathcal I)\subseteq \mathbb V_w(\mathcal J)$ for all
$w\in\C^m$. Now let $w\neq w_0$. For $f\in\mathcal I$ and $q\in\mathbb V_w(\mathcal J)$, we show that
$q(D)f\big{|}_{w}=0$ which implies $q$ must be in $\mathbb V_w(\mathcal I)$.

Note that for any $k\in\N$ and $j\in\{1,\ldots,m\}$, $q(D)\{(z_j-w_{0j})^kf\}\big{|}_{w}=0$ as
$(z_j-w_{0j})^kf\in\mathcal J$. This implies
$\d_{l=0}^k(w_j-w_{0j})^l\binom{k}{l}\frac{\partial^{k-l}q}{\partial z_j^{k-l}}(D)f\big{|}_{w}=0$.  Hence we
have
$$(w_j-w_{0j})^kq(D)f\big{|}_{w}=(-1)^k\frac{\partial^kq}{\partial z_j^k}(D)f\big{|}_{w}\mbox{~for~all~}k\in\N
\mbox{~and~}j\in\{1,\ldots,m\}.$$ So, if $w\neq w_0$, then there exists  $i\in\{1,\ldots,m\}$ such that $w_i\neq
w_{0i}$. Therefore, by choosing $k$ large enough with respect to the degree of $q$, we can ensure
$(w_i-w_{0i})^kq(D)f\big{|}_{w}=0$. Thus $q(D)f\big{|}_{w}=0$. For $w=w_0$, we have

$q\in\mathbb V_{w_0}(\mathcal J)$ if and only if $q(D)\{(z_j-w_{0j})f\}\big{|}_{w_0}=0$ for all $f\in\mathcal I$
and $j\in\{1,\ldots,m\}$ if and only if $\frac{\partial q}{\partial z_j}(D)f\big{|}_{w_0}=0$ for all
$f\in\mathcal I$ and $j\in\{1,\ldots,m\}$ if and only if $q\in\mathbb V_{w_0}(\mathcal J)$ if and only if
$\frac{\partial q}{\partial z_j}\in\mathbb V_{w_0}(\mathcal I)$ for all $j\in\{1,\ldots,m\}$ if and only if
$q\in\tilde{\mathbb V}_{w_0}(\mathcal I)$.

This completes the proof of  the lemma.
\end{proof}

We have shown that $\mathbb V_{w_0}(\mathcal I)=\mathbb V_{w_0}(\mathcal M)=\mathbb V_{w_0}(\mathcal S^{\mathcal
M}_{w_0})$. The next Lemma provides a relationship between the characteristic space of $\mathcal J$ at the point
$w_0$ and  the sheaf  $\mathcal S^{\mathcal M}_{w_0}$.

\begin{lem}$\mathbb V_{w_0}(\mathcal J)=\mathbb V_{w_0}(\mathfrak m(\O_{w_0})\mathcal S^{\mathcal M}_{w_0}).$
\end{lem}
\begin{proof} We have $\mathbb V_{w_0}(\mathfrak m(\O_{w_0})\mathcal S^{\mathcal M}_{w_0}) \subseteq\mathbb V_{w_0}(\mathcal
J)$. 
From the previous lemma, it follows that if $q\in\mathbb V_{w_0}(\mathcal J)$, then $q\in\tilde{\mathbb
V}_{w_0}(\mathcal I)$, that is, $\frac{\partial q}{\partial z_j}\in \mathbb V_{w_0}(\mathcal I)= \mathbb
V_{w_0}(\mathcal S^{\mathcal M}_{w_0})$ for all $j\in\{1,\ldots,m\}$.  From \eqref{gp},  it follows that
$q\in\mathbb V_{w_0}(\mathfrak m(\O_{w_0})\mathcal S^{\mathcal M}_{w_0})$.
\end{proof}

Now, we have all the ingredients to prove that we must have equality in \eqref{ineq} for submodules of analytic
Hilbert modules which are obtained as closure of some polynomial ideal.
\begin{prop}
Let $\mathcal M = [\mathcal I]$ be a submodule of an analytic Hilbert module  over $\C[\underline z]$, where
$\mathcal I$ is an ideal in the polynomial ring $\C[\underline z]$. Then $$\sharp\{minimal ~set~ of ~ generators
~for ~\mathcal S^{\mathcal M}_{w_0}\} = \dim \cap_{j=1}^m \ker ({M_j}-w_{0j})^*.$$
\end{prop}
\begin{proof} Let $p\in\mathcal I$ such that $\pi\circ i_{w_0}(p)=0$, that is, $p_{w_0}\in
\mathfrak m(\O_{w_0})\mathcal S^{\mathcal M}_{w_0}$. The preceding Lemma implies $q(D)p\big{|}_{w_0}=0$ for all
$q\in\mathbb V_{w_0}(\mathcal J)$. So, $p\in \mathcal J^e_{w_0}:=\{r\in\C[\underline z]: q(D)p\big{|}_{w_0}=0,
\mbox{~for~ all~}q\in\mathbb V_{w_0}(\mathcal J)\}$. Therefore, from \cite[Corollary 2.1.2]{cg} we have
$p\in\bigcap_{w\in\C^m}\mathcal J^e_{w}=\mathcal J$. Thus $\ker (\pi\circ i_{w_0})=\mathcal J=\mathfrak
m_{w_0}\mathcal I$. Then the map $\pi\circ i_{w_0} : \dim \mathcal I/\mathfrak m_{w_0}\mathcal I\ra
\dim{\mathcal S^{\mathcal M}_{w_0}}/{\mathfrak m_{w_0}\mathcal S^{\mathcal M}_{w_0}}$ is one-one and
we have
$$\dim \mathcal
I/\mathfrak m_{w_0}\mathcal I\leq \dim{\mathcal S^{\mathcal M}_{w_0}}/{\mathfrak m_{w_0}\mathcal S^{\mathcal M}_{w_0}}.$$ Therefore, we have equality in \eqref{ineq}.\end{proof}

\begin{rem}
Corollary \ref{glea1} is immediate from the Theorem \ref{glea} and the  proposition given above.
\end{rem}

\begin{rem} 
In the paper \cite{dg}, it is proved that if $\mathcal M$ is a closure of an ideal in the polynomial ring and
$w_0\in V(\mathcal I)$ is a smooth point then,
$$\dim\cap_{i=1}^m\ker (M_j-w_{0j})^*=\left\{
\begin{array}{ll}
1 & \hbox{\rm{for} $w_0\notin V(\mathcal I)\cap\Omega$;} \\
\rm{codimension ~of~} \textit{V}(\mathcal I) & \hbox{\rm{for} $w_0\in V(\mathcal I)\cap\Omega$.}
\end{array}
\right.$$ This can be easily derived from the Proposition given above. In the course of the proof of the main
theorem in \cite{dg},  a change of variable argument is used to show that one may assume without of loss of
generality that the stalk at $w_0$
is generated by the co-ordinate functions $z_1,\ldots,z_r$, where $r$ is the co-dimension of $V(\mathcal I)$.
Therefore, the number of minimal generators for the stalk at a smooth point is equal to the codimension of
$V(\mathcal I)$. It now follows from the Proposition that the dimension of the joint kernel at a smooth point is
equal to the co-dimension of $V(\mathcal I)$.
\end{rem}

\section{Bergman space privilege}\label{putinar}
Fix two positive integer $p,q$. The division problem asks if the solution $u\in\O(\Omega)^q$ to the linear
equation $ A u = f $ must belong to $L^2_a(\Omega)^q$ if $f \in L^2_a(\Omega)^p$  and the matrix $A\in
M_{p,q}(\O(\overline{\Omega}))$ of analytic functions defined in a neighborhood of $\overline{\Omega}$ are
given. Two independent steps are necessary to understand the nature of the Division problem.

First, the solution $u$ may not be unique, simply due to the non-trivial relations among the columns of the
matrix $A$. This difficulty is clarified by homological algebra: at the level of coherent analytic sheaves,
$\mathfrak N = {\rm coker}(A: \O|_{\overline\Omega}^p \longrightarrow \O|_{\overline\Omega}^q)$ admits a finite
free resolution
\begin{equation}\label{sheafresolution}
0\rightarrow \O|_{\overline\Omega}^{n_p}\xrightarrow{d_p}\cdots\rightarrow
\O|_{\overline\Omega}^{n_1}\xrightarrow{d_1} \O|_{\overline\Omega}^{n_0}\rightarrow \mathfrak N \rightarrow 0,
\end{equation}
where $n_1 = p, n_0 = q$ and $d_1 = A$. The existence of such a resolution is assured by the analogue of Hilbert
syzygies theorem in the analytic context, see for instance \cite{GR}.

The second step, of circumventing the non-existence of boundary values for Bergman space functions, is resolved
by a canonical quantization method, that is, by passing to the algebra of Toeplitz operators with continuous
symbol on $L^2_a(\Omega)$.  We import below, from the well understood theory of Toeplitz operators on domains of
$\C^m$, a crucial criterion for a matrix of Toepliz operators to be Fredholm (cf. \cite{Sun,V}).

Assume that the analytic matrix $A(z)$ is defined on a neighborhood of $\overline{\Omega}$. One proves by
standard homological techniques that every free, finite type resolution of the analytic coherent sheaf $
\mathfrak N= {\rm coker}(A: \O|_{\overline\Omega}^p \longrightarrow \O|_{\overline\Omega}^q)$ induces at the
level of the Bergman space $L^2_a(\Omega)$ an exact complex, see \cite{D0}. The similarity between the two
resolutions given above are not accidental, as it will be revealed in the next theorem. After understanding the
disc-algebra privilege on a strictly convex domain \cite{PS}, the statement of Theorem \ref{Privilege} is not
surprising.

\begin{proof}[Proof of Theorem \ref{Privilege}] The proof is very similar to the one of the disk algebra case \cite{PS}, and we only sketch below
the main ideas. Assume that the resolution \ref{Bergmanresolution} exists and that the last arrow has closed
range. The exactness at each degree of the resolution is equivalent to the invertibility of the Hodge operator:
$$ d_k^\ast d_k + d_{k+1} d_{k+1}^\ast : L^2_a(\Omega)^{n_k}
\longrightarrow L^2_a(\Omega)^{n_k},\ \ \ 1 \leq k \leq p,$$ where we put $d_{p+1} = 0$. To be more specific:
the condition $\ker [d_k^\ast d_k + d_{k+1} d_{k+1}^\ast] = 0$ is equivalent to the exactness of the complex at
stage $k$, implying hence that ${\rm ran} (d_{k+1})$ is closed. In addition, if the range of $d_k$ is closed,
then, and only then, the self-adjoint operator $d_k^\ast d_k + d_{k+1} d_{k+1}^\ast$ is invertible.

Since the boundary of $\Omega$ is smooth, the commutator $[T_f,T_g]$ of two Toeplitz operators acting on the
Bergman space and with continuous symbols $f,g \in C(\overline{\Omega})$ is compact, see for details and
terminology \cite{B,Sun,V}. Consequently for every $k$, $d_k^\ast d_k + d_{k+1} d_{k+1}^\ast$ is, modulo compact
operators, a $n_k \times n_k$ matrix of Toeplitz operators with symbol
$$ d_k(z)^\ast d_k(z) + d_{k+1}(z) d_{k+1}(z)^\ast, \ \ w \in
\overline{\Omega},$$ where the adjoint is now taken with respect to the canonical inner product in $\C^{n_k}$.
According to a main result of \cite{B}, or \cite{V,Sun}, if the Toeplitz operator $d_k^\ast d_k + d_{k+1}
d_{k+1}^\ast$ is Fredholm, then its matrix symbol is invertible. Hence
$$ \ker [d_k(z)^\ast d_k(z) + d_{k+1}(z) d_{k+1}(z)^\ast] = 0, \ \ 1
\leq k \leq p.$$
Thus, for every $z \in \partial \Omega$,
$$ {\rm rank} A(z) = \dim {\rm coker} (d_1(w)) =
n_0-n_1+n_2-...+(-1)^p n_p.$$

To prove the other implication, we rely on the disk algebra privilege criterion obtained in the note \cite{PS}.
Namely, in view of Theorem 2.2 of \cite{PS}, if the rank of the matrix $A(z)$ does not jump for $z$ belonging to
the boundary of $\Omega$, then there exists a resolution of $\mathbf N = {\rm coker} \ A : \A(\Omega)^p
\longrightarrow \A(\Omega)^q$ with free, finite type $\A(\Omega)$-modules:
\begin{equation}\label{diskalgresolution}
0\rightarrow \A(\Omega)^{n_p}\xrightarrow{d_p}\cdots\rightarrow \A(\Omega)^{n_1}\xrightarrow{d_1}
\A(\Omega)^{n_0}\rightarrow \mathbf N\rightarrow 0.
\end{equation}
As before, we denote $d_1 = A$. We have to prove that the induced complex (\ref{Bergmanresolution}), obtained
after applying (\ref{diskalgresolution}) the functor $\otimes_{\A(\Omega)} L^2_a(\Omega)$, remains exact and the
boundary operator $d_1$ has closed range.

For this, we ``glue'' together local resolutions of ${\rm coker} A$ with the aid of Cartan's lemma of invertible
matrices, as originally explained in \cite{D1}, or in \cite{PS}. For points close to the boundary of $\Omega$,
such a resolution exists by the local freeness assumption, while in the interior, in neighborhoods of the points
where the rank of the matrix $A$ may jump, they exist by Douady's privilege on polydiscs. This proves that the
Hilbert analytic module $\mathcal N={\rm coker}( A : L^2_a(\Omega)^p \longrightarrow L^2_a(\Omega)^q)$ is privileged with
respect to the Bergman space.

As for assertion c), we simply remark that it is equivalent to the injectivity of the restriction map
$$ {\rm coker}(A : L^2_a(\Omega)^p
\longrightarrow L^2_a(\Omega)^q) \longrightarrow {\rm coker}(A : \O(\Omega)^p \longrightarrow \O(\Omega)^q).$$
The last co-kernel is always Hausdorff in the natural quotient topology as the global section space of a
coherent analytic sheaf.

The only place in the proof where the convexity of $\Omega$ is needed, is to ensure that, if the resolution
\ref{Bergmanresolution} exists, then the induced complex at the level of sheaf models
$$
0\rightarrow {\widehat{L^2_a}(\Omega)}^{n_p}\xrightarrow{d_p}\cdots\rightarrow
{\widehat{L^2_a}(\Omega)}^{n_1}\xrightarrow{d_1} {\widehat{L^2_a}(\Omega)}^{n_0}\rightarrow \widehat{\mathcal
N}\rightarrow 0,
$$
is exact. For a proof see \cite{PS}.
\end{proof}

\begin{rem}It is worth mentioning that for non-smooth domains $\Omega$ in $\C^m$ the above result is not true. For
instance $\A(\Omega)$-privilege for a poly-domain $\Omega$ was fully characterized by Douady \cite{D1}. On the
other hand, even for smooth boundaries, the privilege with respect to the Fr\'echet algebra $\O(\Omega) \cap
C^\infty(\overline{\Omega})$ seems to be quite intricate and definitely different than the Bergman space or disk
algebra privileges, as indicated by an observation of Amar \cite{A}. \end{rem}

\begin{cor}
Coker$~[(\varphi_1,\ldots,\varphi_m) : L^2_a(\Omega)^m\ra L^2_a(\Omega)^m]$ is privileged if and only if the
analytic functions $(\varphi_1,\ldots,\varphi_m)$ have no common zero on the boundary.
\end{cor}
\begin{proof}
No common zero of the functions  $\varphi_1,\ldots,\varphi_m$ lies on the boundary of $\Omega$.  Therefore, the
matrix $(\varphi_1,\ldots,\varphi_m)$ is of full rank $1$ on the boundary of $\Omega$.
\end{proof}
For many semi-Fredholm Hilbert module such as the Hardy space on $\Omega$, the result given above, remains true
(\cite{D0,D1}).

Since the restriction to an open subset $\Omega_0 \subseteq \Omega$ does not change the equivalence class of a
module in $\mathfrak B_1(\Omega)$, we can always assume, without loss of generality, that the domain $\Omega$ is
pseudoconvex in our context. For $w_0\in\Omega$, the $m$-tuple $(z_1-w_{01},\ldots,z_m-w_{0m})$ has no common
zero on the boundary of $\Omega$. We have pointed out, in Section \ref{int}, that if for $f\in \mathcal M$ the
equation $f=\d_{i=1}^m(z_i-w_{0i})f_i$ admits a solution $(f_1, \ldots, f_m)$ in $\O(\Omega)^m$ and if the
module $\mathcal M$ is privileged, then the solution is in $\mathcal M^m$. This shows that $f\in\mathcal
M^{(w_0)}$. Thus for Hilbert modules which are privileged, we have
$$\sharp \{minimal ~generators ~for ~S^{\mathcal M}_{w}\} = \mbox{dim~} \cap_{j=1}^m \ker ({M_j}-w_{0j})^*.$$

In accordance with the terminology of local spectral theory, see \cite{EP}, we isolate the following
observation.

\begin{cor}\label{Sheafmodel}
Assume that the analytic module $\mathcal N= {\rm coker}(A : L^2_a(\Omega)^p \longrightarrow L^2_a(\Omega)^q)$
is Hausdorff, where $A$ and $\Omega$ are as in the Theorem. Then $\mathcal N$ is a Hilbert analytic
quasi-coherent module, and for every Stein open subset $U$ of $\C^m$, the associated sheaf model is
$$ \widehat{\mathcal N}(U) = \O(U) \hat{\otimes}_{\A(\Omega)} \mathcal N =$$ $$
       {\rm coker}(A: \H(U)^p \longrightarrow \H(U)^q) = $$ $$
       {\rm coker }(z-w: \O(U) \hat{\otimes}\mathcal N^m \longrightarrow \O(U) \hat{\otimes} \mathcal N),$$
       where $\H$ denotes the sheaf model of the Bergman space.
       \end{cor}

\begin{rem}\label{freeres}
      We recall that (see \cite{EP})
      $$ \H(U) = \{ f \in \O(U \cap \Omega); \ \| f \|_{2,K} < \infty, \ \ K \ \ {\rm compact\ in} \ \ U\}.$$
      Since $\H|_\Omega = \O|_\Omega$ we infer that the restriction $\widehat{\mathcal N}|_\Omega$ is a
      coherent sheaf, with finite free resolution
      $$  0\rightarrow \O|_\Omega^{ n_p}\xrightarrow{d_p}\cdots\rightarrow
\O|_\Omega^{n_1}\xrightarrow{d_1} \O|_\Omega^{n_0}\rightarrow \widehat{\mathcal N}|_\Omega \rightarrow 0.$$
\end{rem}

\section{Coincidence of sheaf models}\label{co.sh}

Besides the expected relaxations of the main result above, for instance from convex to pseudoconvex domains, a
natural problem to consider at this stage is the classification of the analytic Hilbert modules $\mathcal N =
{\rm coker} (A : L^2_a(\Omega)^m \longrightarrow L^2_a(\Omega)^n)$ appearing in the Theorem \ref{Privilege}
above. This question fits into the framework of quasi-free Hilbert modules introduced in \cite{dm1}. That the
resulting parameter space is wild, there is no doubt, as all Artinian modules $M$ (over the polynomial algebra)
supported by a fix point $w_0 \in \Omega$ enter into our discussion. Specifically, we can take
$$\mathcal N = {\rm coker} ((\varphi_1,...,\varphi_p): L^2_a(\Omega)^p \longrightarrow
L^2_a(\Omega)),$$ where $\varphi_1,...,\varphi_p$ are polynomials with the only common zero $\{ w_0\}$. Then in
virtue of Theorem 2.1, the analytic module $M$ is finite dimensional and privileged with respect to the Bergman
space $L^2_a(\Omega)$. An algebraic reduction of the classification of all finite co-dimension analytic Hilbert
modules of the Bergman space associated of a smooth, strictly convex domain can be found in \cite{put,P2}.

In order to better relate the Cowen-Douglas theory to the above framework, we consider together with the map
$$A : L^2_a(\Omega)^p
\longrightarrow L^2_a(\Omega)^q$$ whose cokernel was supposed to be Hausdorff,  the dual, anti-analytic map
$$ A^\ast : L^2_a(\Omega)^q
\longrightarrow L^2_a(\Omega)^p.$$ It is the linear system, in the terminology of Grothendieck \cite{SC} or
\cite{Fi}, with its associated Hermitian structure induced from the embedding into Bergman space,
$$ {\ker} A^*(z) \subset L^2_a(\Omega)^q, \ \ z \in \Omega,$$
which was initially considered in operator theory, see \cite{dp}.

 Traditionally one
works with the torsion-free module
$$ \mathcal M = {\rm ran}  (A : L^2_a(\Omega)^p
\longrightarrow L^2_a(\Omega)^q),$$ rather than the cokernel $\mathcal N$ studied in the previous section. A
short exact sequence relates the two modules:
$$ 0 \longrightarrow \mathcal M \longrightarrow L^2_a(\Omega)^q \longrightarrow \mathcal N \longrightarrow 0.$$

\begin{prop} Assume, in the conditions of Theorem \ref{Privilege}, that
the range $\mathcal M$ of the module map $A$ is closed. Then $\mathcal M$ is an analytic Hilbert quasi-coherent
module, with associated sheaf model
$$ \widehat{\mathcal M}(U) = {\rm ran}(A : \H(U)^p \longrightarrow \H(U)^q),$$
for every Stein open subset $U$ of $\C^m$.

In particular, for every point $w_0 \in \Omega$, there are finitely many elements $g_1,...,g_d \in \mathcal M
\subset L^2_a(\Omega)^q$, such that the stalk $\widehat{\mathcal M}_{w_0}$ coincides with the $\O_{w_0}$-module
generated in $\O_{w_0}^q$ by $g_1,...,g_d $.
\end{prop}

\begin{proof} The first assertion follows from the main result of the previous section and the yoga
of quasi-coherent sheaves. In particular we obtain an exact complex of coherent analytic sheaves
$$ 0 \longrightarrow \widehat{\mathcal M}|_\Omega \longrightarrow \O_\Omega^n  \longrightarrow \widehat{\mathcal N}|_\Omega\longrightarrow 0.$$

 For the proof of the second assertion, recall that the quasi-coherence of $\mathcal  M$
yields a finite presentation, derived from the associated Koszul complex,
$$ \O_{w_0}^m \hat{\otimes} \mathcal M \stackrel{z-w}{\longrightarrow} \O_{w_0} \hat{\otimes} \mathcal M \longrightarrow \widehat{\mathcal M}_{w_0} \longrightarrow 0.$$
By evaluating the presentation at $w=w_0$, we obtain the exact complex
$$ \mathcal M^m \stackrel{z-w_0}{\longrightarrow} \mathcal M \longrightarrow \widehat{\mathcal M}(w_0) \longrightarrow 0.$$
Above we denote by $w=(w_1,...,w_m)$ the $m$-tuple of local coordinates in the ring $\O_{w_0}$, while
$z=(z_1,...,z_d)$ stands for the $d$-tuple of coordinate functions in the base space of the Hilbert module
$L^2_a(\Omega)$.

By coherence, $\dim \widehat{\mathcal M}(w_0) < \infty$, and it remains to choose the k-tuple of elements
$g=(g_1,\ldots,g_d)$ as a basis of the ortho-complement of ${\rm ran} (z-w_0: \mathcal M^m {\longrightarrow}
\mathcal M)$. Then the map
$$ \O_{w_0}^m \hat{\otimes} (\mathcal M \oplus \C^m) \stackrel{z-w,g}{\longrightarrow} \O_{w_0} \hat{\otimes} \mathcal M$$
is onto. Consequently, the functions $g_1,...,g_d$ generate $\widehat{\mathcal M}_{w_0}$ as a submodule of
$\O_{w_0}^q$. As a matter of fact the same functions will generate $\widehat{\mathcal M}_w$ for all points $w$
belonging to a neighborhood of $w_0$.
\end{proof}

\begin{cor} Under the assumptions of the Proposition, the restriction to $\Omega$ of the
sheaf model $\widehat{\mathcal M} = \stackrel{\widehat{\phantom{ran}}}{{\mathrm{ran}}\,A}$ coincides with the
analytic subsheaf of $\O^q$ generated by all functions $f|_\Omega,\, f \in \mathcal M$.
\end{cor}

The dual picture emerges easily: let $w_0$ be a fixed point of $\Omega$, under the assumptions of Theorem
\ref{Privilege}, the map $A_{w_0}(z):=(z_1-w_{01}, \ldots ,z_m-w_{0m}) :\mathcal M^m \longrightarrow \mathcal M$
has finite dimensional cokernel. Choose a basis $v_1,...,v_\ell$ of $\ker A_{w_0}(z)^\ast$ and denote by $P_w$
the orthogonal projection onto $\ker A_w(z)^\ast$. Then for $w$ belonging to a  small enough open neighborhood
$V$ of $w_0$, the elements $P_w(v_1),...,P_w(v_\ell)$ generate $\ker A_{w}(z)^\ast$ as a vector space, but they
need not remain linearly independent on $V$. Never the less, starting with a module $\mathcal M$ in $\mathfrak
B_1(\Omega)$, in the next section, we will provide a  construction of a holomorphic Hermitian vector bundle
$E_\mathcal M$ on $V$.
\section{Classification of Hilbert modules and curvature invariants}\label{inv}
Let $\mathcal M$ be a Hilbert module in $\mathfrak B_1(\Omega)$ and $w_0 \in \Omega$ be fixed. The vectors
$K^{(i)}_w \in \mathcal M,\, 1\leq i\leq d$, produced in part (ii) of the decomposition theorem \ref{decomp} are
independent in some small neighborhood, say $\Omega_0$  of $w_0$. However, while the choice of these vectors is
not canonical, in general, we provide below a recipe for finding the vectors $K^{(i)}_w,\, 1\leq i\leq d$,
satisfying
$$
K(\cdot,w) = g^0_1(w)K^{(1)}_w + \cdots + g^0_n(w)K^{(d)}_w,\, w\in \Omega_0
$$
following \cite{cs}.
We note that $\mathfrak m_w\mathcal M$ is a closed submodule of $\mathcal M$. We assume that we have equality in
\eqref{ineq} for the module $\mathcal M$ at the point $w_0\in\Omega$, that is, $\mbox{span}_{\C}\{K^{(i)}_{w_0}:
1\leq i\leq d\} = \dim \cap_{j=1}^m \ker ({M_j}-w_{0j})^*.$

Let $D_{(\mathbf M-w)^*}= V_{\mathbf M}(w)|D_{(\mathbf M-w)^*}|$ be the polar decomposition of $D_{(\mathbf
M-w)^*}$, where $|D_{(\mathbf M-w)^*}|$ is the positive square root of the operator $\big (D_{(\mathbf
M-w)^*}\big )^* D_{(\mathbf M-w)^*}$ and $V_{\mathbf M}(w)$ is the partial isometry mapping $\big (\ker
D_{(\mathbf M-w)^*}\big )^{\perp}$ isometrically onto $\mathrm{ran} D_{(\mathbf M-w)^*}$. Let $Q_{\mathbf M}(w)$
be the positive operator:
$$Q_{\mathbf M}(w) \big{|}_{\ker D_{(\mathbf M-w)^*}}=0\mbox{~and~} Q_{\mathbf M}(w) \big{|}_{(\ker D_{(\mathbf M-w)^*})^{\perp}}=
\big(|D_{(\mathbf M-w)^*}|\big{|}_{(\ker D_{(\mathbf M-w)^*})^{\perp}}\big)^{-1}.$$ Let $R_{\mathbf M}(w):
\mathcal M\oplus\cdots\oplus\mathcal M\ra\mathcal M$ be the operator  $R_{\mathbf M}(w)=Q_{\mathbf
M}(w)V_{\mathbf M}(w)^*$. The two equations, involving the operator $D_{(\mathbf M-w)^*}$, stated below are
analogous to the semi-Fredholmness property of a single operator (cf. \cite[Proposition 1.11]{cd}):
\begin{eqnarray}
R_{\mathbf M}(w)D_{(\mathbf M-w)^*} & = &I- P_{\ker D_{(\mathbf M-w)^*}} \label{cs1}\\ D_{(\mathbf M-w)^*}
R_{\mathbf M}(w) &=& P_{\mathrm{ran}\, D_{(\mathbf M-w)^*}} \label{cs2},
\end{eqnarray}
where $P_{\ker D_{(\mathbf M-w)^*}}, P_{\mathrm{ran} D_{(\mathbf M-w)^*}}$ are orthogonal projection onto $\ker
D_{(\mathbf M-w)^*}$ and $\mathrm{ran} D_{(\mathbf M-w)^*}$ respectively. Consider the operator
$$
P(\bar w,\bar w_0)=I-\{I-R_{\mathbf M}(w_0)D_{\bar w-\bar w_0}\}^{-1}R_{\mathbf M}(w_0) D_{(\mathbf M-w)^*},\,
w\in B(w_0; {\parallel R(w_0)\parallel}^{-1}),$$ where $B(w_0; {\parallel R(w_0)\parallel}^{-1})$ is the ball of
radius ${\parallel R(w_0)\parallel}^{-1}$ around $w_0$. Using the equations \eqref{cs1} and \eqref{cs2} given
above, we write
\begin{eqnarray}\label{cs3}
P(\bar w,\bar w_0)=\{I-R_{\mathbf M}(w_0)D_{\bar w-\bar w_0}\}^{-1}P_{\ker D_{(\mathbf M-w)^*}},
\end{eqnarray}
where $D_{\bar w-\bar w_0}f=((\bar w_1-\bar w_{01})f_1,\ldots,(\bar w_m-\bar w_{0m})f_m)$. From definition of
$P(\bar w, \bar w_0)$ it follows that $P(\bar w, \bar w_0)P_{\ker D_{(\mathbf M-w)^*}}=P_{\ker D_{(\mathbf
M-w)^*}}$ which implies $ \ker D_{(\mathbf M-w)^*} \subset \mathrm{ran} P(\bar w,\bar w_0)$ for $w\in
\Delta(w_0;\varepsilon)$. Consequently $K(\cdot,w)\in \mathrm{ran} P(\bar w,\bar w_0)$ and therefore $$K(\cdot,
w)=\d_{i=1}^d\ov{a_i(w)}P(\bar w,\bar w_0)K^{(i)}_{w_0},$$ for some complex valued functions $a_1,\ldots, a_d$
on $\Delta(w_0;\varepsilon)$. We will show that the functions $a_i,\, 1\leq i\leq d$, are holomorphic and their
germs form a minimal set of generators for $S^{\mathcal M}_{w_0}$. Now
$$R_{\mathbf M}(w_0)D_{\bar w - \bar w_0}K(\cdot,w)=R_{\mathbf M}(w_0)D_{(\mathbf M-w_0)^*}K(\cdot,w)= (I-P_{\ker
D_{(\mathbf M-w_0)^*}})K(\cdot,w).$$ Hence we have,
$$\{I-R_{\mathbf M}(w_0)D_{\bar w-\bar w_0}\}K(\cdot,w)=P_{\ker D_{(\mathbf M-w_0)^*}}K(\cdot,w).$$ Since $K(\cdot,w)\in  \mathrm{ran} P(\bar w,\bar w_0)$,
we also have $${P(\bar w,\bar w_0)}^{-1}K(\cdot,w)=P_{\ker D_{(\mathbf M-w_0)^*}}K(\cdot,w).$$  Let $v_1,\ldots
,v_d$ be the orthonormal basis for $\ker D_{(\mathbf M-w_0)^*}$. Let $g_1,\ldots,g_d$ denotes the minimal set of
generators for the stalk at $\mathcal S^{\mathcal M}_{w_0}$.  Then there exist a neighborhood $U$, small enough
such that $v_j=\d_{i=1}^dg_if_i^j$, $1\leq j\leq d$, and for some holomorphic functions $f_i^j,1\leq i,j\leq d$,
on $U$. We then have
\begin{eqnarray*}
{P(\bar w,\bar w_0)}^{-1}K(\cdot,w) &=& P_{\ker D_{(\mathbf M-w_0)^*}}K(\cdot,w)~=~ \d_{j=1}^d\langle K(\cdot,
w),v_j\rangle v_j\\ &=&
\d_{j=1}^d\langle K(\cdot, w),\d_{i=1}^dg_if_i^j\rangle v_j  ~=~ \d_{i=1}^d\d_{j=1}^d\ov{g_i(w)f_i^j(w)}v_j\\
&=& \d_{i=1}^d\ov{g_i(w)}\{\d_{j=1}^d\ov{f_i^j(w)}v_j\}.
\end{eqnarray*}
So $K(z,w)=\d_{i=1}^d\ov{g_i(w)}\{\d_{j=1}^d\ov{f_i^j(w)}P(\bar w,\bar w_0)v_j(z)\}$. Let
$$\tilde K^{(i)}_w=\d_{j=1}^d\ov{f_i^j(w)}P(\bar w,\bar w_0)v_j.$$ Since the vectors $K^{(i)}_{w_0},\, 1\leq i\leq d$
are uniquely determined as long as $g_1,\ldots ,g_d$ are fixed and $P(\bar w_0,\bar w_0)=P_{\ker D_{(\mathbf
M-w_0)^*}}$, it follows that $K^{(i)}_{w_0}=\tilde K^{(i)}_{w_0}=\d_{j=1}^d\ov{f_i^j(w_0)}v_j,\, 1\leq i\leq d$.
Therefore, the $d\times d$ matrix $(\ov{f_i^j(w_0)})_{i,j=1}^d$ has a non-zero determinant. As $Det~
(\ov{f_i^j(w)})_{i,j=1}^d$ is an anti-holomorphic function, there exist a neighbourhood of $w_0$, say
$\Delta(w_0;\varepsilon), \varepsilon>0$, such that $Det ~(\ov{f_i^j(w)})_{i,j=1}^d\neq0$ for all
$w\in\Delta(w_0;\varepsilon)$. The set of vectors $\{P(\bar w,\bar w_0)v_j\}_{j=1}^n$ is linearly independent
since $P(\bar w,\bar w_0)$ is injective on $\ker D_{(\mathbf M-w_0)^*}$. Let $(\alpha_{ij})_{i,j=1}^d =
\{(\ov{f_i^j(w_0)})_{i,j=1}^d\}^{-1}$, in consequence, $v_j=\d_{l=1}^d\alpha_{jl}K^{(l)}_{w_0}$. We then have
\begin{eqnarray*}K(\cdot,w) &=& \d_{i=1}^d\ov{g_i(w)}\{\d_{j=1}^d\ov{f_i^j(w)}P(\bar w,\bar w_0)(\d_{l=1}^d\alpha_{jl}
K^{(l)}_{w_0})\}\\
&=& \d_{l=1}^d\{\d_{i,j=1}^d\ov{g_i(w)}\ov{f_i^j(w)}\alpha_{jl}\}P(\bar w,\bar
w_0)K^{(l)}_{w_0}).\end{eqnarray*} Since the matrices $(\ov{f_i^j(w)})_{i,j=1}^d$ and $(\alpha_{ij})_{i,j=1}^d$
are invertible, the functions $$a_l(z)=\d_{i,j=1}^d{g_i(z)}{f_i^j(z)}\alpha_{jl}, ~1\leq l\leq d,$$ form a
minimal set of generators for the stalk $S^{\mathcal M}_{w_0}$ and hence we have the canonical decomposition,
$$K(\cdot, w)=\d_{i=1}^d\ov{a_i(w)}P(\bar w,\bar w_0)K^{(i)}_{w_0}.$$

Let $\mathcal P_w = \mathrm {ran} P(\bar w, \bar w_0)P_{\ker D_{(\mathbf M-w_0)^*}}$ for $w\in B(w_0; {\parallel
R_{\mathbf M}(w_0)\parallel}^{-1})$. Since $P(\bar w, \bar w_0)$ restricted to the $\ker D_{(\mathbf M-w_0)^*}$
is one-one, $\dim \mathcal P_w$ is constant for $w\in B(w_0; {\parallel R_{\mathbf M}(w_0)\parallel}^{-1})$.
Thus to prove Lemma \ref{const}, we will show that $\mathcal P_w = \ker \mathbb P_0 D_{(\mathbf M - w)^*}$.

\begin{proof}[Proof of Lemma \ref{const}]
From \cite[pp. 453]{cs}, it follows that $\mathbb P_0 D_{(\mathbf M - w)^*}P(\bar w, \bar w_0)=0$. So, $\mathcal
P_w\subseteq \ker \mathbb P_0 D_{(\mathbf M - w)^*}$. Using \eqref{cs1} and \eqref{cs2}, we can write
\begin{eqnarray*}
\mathbb P_0 D_{(\mathbf M - w)^*} &=& D_{(\mathbf M - w_0)^*}R_{\mathbf M}(w_0)\{D_{(\mathbf M - w_0)^*} -
D_{(\bar w -
\bar w_0)}\}\\
&=& D_{(\mathbf M - w_0)^*}\{I - P_{\ker D_{(\mathbf M-w_0)^*}} - R_{\mathbf M}(w_0)D_{(\bar w - \bar w_0)}\}\\
&=& D_{(\mathbf M - w_0)^*}\{I  - R_{\mathbf M}(w_0)D_{(\bar w - \bar w_0)}\}.
\end{eqnarray*}
Since $\{I  - R_{\mathbf M}(w_0)D_{(\bar w - \bar w_0)}\}$ is invertible for $w\in B(w_0; {\parallel R_{\mathbf
M}(w_0)\parallel}^{-1})$, we have
$$\dim\mathcal P_w = \dim D_{(\mathbf M - w_0)^*} \geq \dim\ker \mathbb P_0 D_{(\mathbf M - w)^*}. $$
This completes the proof.
\end{proof}
From the construction of the operator $P(\bar w, \bar w_0)$, it follows that, the association $w\ra\mathcal P_w$
forms a Hermitian holomorphic vector bundle of rank $m$ over $\Omega_0^* = \{\bar z: z\in\Omega_0\}$ where
$\Omega_0=B(w_0; {\parallel R_{\mathbf M}(w_0)\parallel}^{-1})$. Let $\mathcal P$ denote this Hermitian
holomorphic vector bundle.

\begin{proof}[Proof of Theorem\ref{csb}] Since $\mathcal M$ and $\mathcal {\tilde M}$ are equivalent Hilbert modules, there exist a unitary
$U:\mathcal M\ra\tilde{\mathcal M}$  intertwining the adjoint of the module multiplication, that is, $U{M_j}^*=
{{\tilde M}_j}^*U$, $1\leq j\leq m$. Here ${\tilde M}_j$ denotes the multiplication by co-ordinate function
$z_j, 1\leq j\leq m$ on $\tilde{\mathcal M }$. It is enough to show that $UP(\bar w,\bar w_0)=\tilde P(\bar
w,\bar w_0)U$ for $w\in B(w_0; {\parallel R_{\mathbf M}(w_0)\parallel}^{-1})$.

Let $\mid D_{\mathbf M^*}\mid=\{\d_{j=1}^mM_j{M_j}^*\}^{\frac {1}{2}}$, that is, the positive square root of
$(D_{\mathbf M^*})^*D_{\mathbf M^*}$. We have
$$\d_{j=1}^mM_j{M_j}^*=U^*(\d_{j=1}^m{\tilde M_j}{\tilde
M_j}^*)U=(U^*\mid D_{\tilde {\mathbf M}^*}\mid U)^2.$$ Clearly, $\mid D_{{\mathbf M}^*}\mid=U^*\mid D_{\tilde
{\mathbf M}^*}\mid U$. Similar calculation gives $\mid D_{(\mathbf M-w_0)^*}\mid=U^*\mid D_{(\tilde {\mathbf
M}-w_0)^*}\mid U$. Let $P_i:\mathcal M\oplus\mathcal M\cdots \oplus\mathcal M(\mbox{~m~ times})\lo\mathcal M$ be
the orthogonal projection on the $i$-th component.  In this notation, we have $P_jD_{{\mathbf M}^*}={M_j}^*,
1\leq j\leq m$. Then,
\begin{eqnarray*}
{\tilde P}_jD_{(\tilde {\mathbf M}-w_0)^*} &=&
UP_jD_{(\mathbf M-w_0)^*}U^* = UP_jV_{\mathbf M}(w_0)U^*U\mid D_{(\mathbf M-w_0)^*}\mid U^*\\
&=& UP_jV_{\mathbf M}(w_0)U^*\mid D_{(\tilde {\mathbf M}-w_0)^*}\mid.
\end{eqnarray*}
But ${\tilde P}_jD_{(\tilde {\mathbf M}-w_0)^*}={\tilde P}_j{V_{\tilde {\mathbf M}}(w_0)}\mid D_{(\tilde
M-w_0)^*}\mid$. The uniqueness of the polar decomposition implies that ${\tilde P}_j{V_{\tilde {\mathbf
M}}(w_0)}=UP_jV_{\mathbf M}(w_0)U^*, \, 1\leq j\leq m$. It follows that $Q_{\tilde M}(w_0)=UQ_M(w_0)U^*$.

Note that ${P_j}^*:\mathcal M\lo\mathcal M\oplus\cdots \oplus\mathcal M$ is given by ${P_j}^*h=(0,\ldots
,h,\ldots ,0)$, $h\in\mathcal M, 1\leq j\leq m$. So we have, ${ V_{\tilde {\mathbf M}}(w_0)}^*{{\tilde
P}_j}^*=U{V_{\mathbf M}(w_0)}^*{P_j}^*U^*$, $1\leq j\leq m$. Let $\tilde D_{\bar w}:\mathcal M\lo\mathcal
M\oplus\cdots \oplus\mathcal M$ be the operator: $\tilde D_{\bar w}f = (\bar w_1f,\ldots, \bar w_mf), \, f\in
\tilde{\mathcal M}$. Clearly, $\tilde D_{\bar w}=UD_{\bar w}U^*$, that is, $U^*{\tilde P}_j\tilde D_{\bar
w}=P_jD_{\bar w}U^*$, $1\leq j\leq m$. Finally,
\begin{eqnarray*}
\lefteqn{R_{\tilde {\mathbf M}}(w_0)\tilde D_{\bar w-\bar w_0}}\\ &=& Q_{\tilde {\mathbf M}}(w_0){V_{\tilde
{\mathbf M}}(w_0)}^*\tilde D_{\bar w - \bar w_0} ~=~Q_{\tilde {\mathbf M}}(w_0){V_{\tilde {\mathbf M}}(w_0)}^*
({\tilde P}_1 \tilde
D_{\bar w - \bar w_0},\ldots ,{\tilde P}_m\tilde D_{\bar w - \bar w_0}) \\
&=& Q_{\tilde {\mathbf M}}(w_0){V_{\tilde {\mathbf M}}(w_0)}^*(\d_{j=1}^m {{\tilde P}_j}^*{\tilde P}_j\tilde
D_{\bar w - \bar w_0})
\\ &=& Q_{\tilde {\mathbf M}}(w_0)U{V_{\mathbf M}(w_0)}^*(\d_{j=1}^m{P_j}^*U^*{\tilde P}_j
\tilde D_{\bar w - \bar w_0}) \\
&=& U Q_{\mathbf M}(w_0){V_{\mathbf M}(w_0)}^*(\d_{j=1}^m{P_j}^*P_jD_{\bar w -\bar w_0}U^*) ~=~U Q_{\mathbf
M}(w_0){V_{\mathbf M}(w_0)}^*D_{\bar w-\bar w_0}U^*\\ &=& U R_{\mathbf M}(w_0) D_{\bar w -\bar w_0} U^*.
\end{eqnarray*} Hence $\{R_{\tilde {\mathbf M}}(w_0)\tilde D_{\bar w-\bar w_0}\}^k=U \{R_{\mathbf M}(w_0) D_{\bar w-\bar w_0}\}^k U^*$ for all $k\in\N$.
From \eqref{cs3}, $P(\bar w,\bar w_0)= \d_{k=0}^{\infty}\{R_{\mathbf M}(w_0) D_{\bar w-\bar w_0}\}^k P_{\ker
D_{(\mathbf M-w_0)^*}}$. Also as $U$ maps $\ker D_{(\mathbf M-w)^*}$ onto $\ker D_{(\tilde M-w)^*}$ for each
$w$, we have in particular, $U P_{\ker D_{(\mathbf M-w_0)^*}}= P_{\ker D_{(\tilde M-w_0)^*}}U$.
 Therefore,
\begin{eqnarray*} \lefteqn{UP(\bar w,\bar w_0) }\\&=& \d_{k=0}^{\infty}U\{R_{\mathbf M}(w_0) D_{\bar w-\bar w_0}\}^k P_{\ker D_{(\mathbf M-w_0)^*}} ~=~
\d_{k=0}^{\infty} \{R_{\tilde {\mathbf M}}(w_0)\tilde D_{\bar w-\bar w_0}\}^k U P_{\ker D_{(\mathbf M-w_0)^*}}\\
&=& \d_{k=0}^{\infty}\{R_{\tilde {\mathbf M}}(w_0) \tilde D_{\bar w-\bar w_0}\}^k P_{\ker D_{(\tilde M-w_0)^*}}
U ~=~ \tilde P(\bar w,\bar w_0)U,\end{eqnarray*} for $w\in B(w_0; {\parallel R_{\mathbf
M}(w_0)\parallel}^{-1})$.
\end{proof}

\begin{rem} For any commuting $m$-tuple $D_{\mathbf T}=(T_1,\ldots,T_m)$ of operator on $\mathcal H$, the construction given above, of the
Hermitian holomorphic vector bundle, provides a unitary invariant, assuming only that $\mathrm{ran} D_{\mathbf T
- w}$ is closed for $w$ in  $\Omega\subseteq \C^m$.  Consequently, the class of this Hermitian holomorphic vector bundle is an invariant for any  semi-Fredholm Hilbert module over $\mathbb C[\underline{z}]$.  
\end{rem}

\section{Examples}\label{ex}
\subsection{\sf \sf The $(\lambda,\mu)$ examples} \label{lm} Let $\mathcal M$ and $\tilde{\mathcal M}$ be two Hilbert modules in $\mathrm B_1(\Omega)$ and $\mathcal I,\, \mathcal J$ be two ideals in $\C[\underline z]$. Let  $\mathcal
M_{\mathcal I}:=[\mathcal I] \subseteq \mathcal M$  (resp. $\tilde{\mathcal M}_{\mathcal J}:=[\mathcal
J]\subset\tilde{\mathcal M}$) denote the closure of $\mathcal I$  in $\mathcal M$ (resp. closure of $\mathcal J$
in $\tilde{\mathcal M}$). Also we let $\dim V(\mathcal I), \dim V(\mathcal J)\leq m-2$. It is then not hard to
see that $\mathcal M_{\mathcal I}$ and $\tilde{\mathcal M}_{\mathcal J}$ are equivalent if and only if $\mathcal
I=\mathcal J$ following the argument in the proof \cite[Theorem 2.10]{as} and using the characteristic space
theory of \cite[Chapter 2]{cg}. Assume $\mathcal M$ and $\tilde{\mathcal M}$ are minimal extensions of the two
modules $\mathcal  M_{\mathcal I}$ and $\tilde{\mathcal M}_{\mathcal I}$ respectively and that $\mathcal
M_{\mathcal I}$ is equivalent to $\tilde{\mathcal M}_{\mathcal I}$. We ask if these assumptions force the
extensions $\mathcal M$ and $\tilde{\mathcal M}$ to be equivalent. The answer for a class of examples is given
below.

For $\lambda, \mu > 0$, let $H^{(\lambda, \mu)}(\D^2)$ be the reproducing kernel Hilbert space on the bi-disc
determined by the positive definite kernel
$$K^{(\lambda, \mu)}(z,w)= \frac{1}{(1-z_1\bar w_1)^{\lambda}(1- z_2\bar w_2)^{\mu}},~ z,w\in\D^2.$$
As is well-known, $H^{(\lambda, \mu)}(\D^2)$ is in $\mathrm{B}_1(\mathbb D^2)$. Let $I$ be the maximal ideal in
$\mathcal C_2$ of polynomials vanishing at $(0,0)$. Let $H_0^{(\lambda, \mu)}(\D^2):=[I]$.  For any other pair
of positive numbers $\lambda^\prime, \mu^\prime$, we let $H_0^{(\lambda^\prime, \mu^\prime)}(\D^2)$  denote the
closure of $I$ in the reproducing kernel Hilbert space  $H^{(\lambda^\prime, \mu^\prime)}(\D^2)$. Let
$K^{(\lambda^\prime, \mu^\prime)}$ denote  the corresponding reproducing kernel.  The modules  $H^{(\lambda,
\mu)}(\D^2)$ and  $H^{(\lambda^\prime, \mu^\prime)}(\D^2)$  are in $\mathrm{B}_1(\mathbb D^2 \setminus
\{(0,0)\})$ but not in $\mathrm{B}_1(\mathbb D^2)$.  So, there is no easy computation to determine when they are
equivalent.  We compute the curvature, at $(0,0)$, of the  holomorphic Hermitian bundle $\mathcal P$ and
$\tilde{\mathcal P}$ of rank $2$ corresponding to the modules $H_0^{(\lambda, \mu)}(\D^2)$ and
$H_0^{(\lambda^\prime, \mu^\prime)}(\D^2)$ respectively.  The calculation of the curvature show that if these modules are
equivalent then $\lambda=\lambda^\prime$ and $\mu=\mu^\prime$, that is, the extensions $H^{(\lambda,
\mu)}(\D^2)$ and  $H^{(\lambda^\prime, \mu^\prime)}(\D^2)$ are then equivalent.

Since $H_0^{(\lambda, \mu)}(\D^2) :=\{f\in H^{(\lambda, \mu)}(\D^2):f(0, 0)=0\}$, the corresponding reproducing
kernel $K^{(\lambda, \mu)}_0$ is given by the formula $$K^{(\lambda, \mu)}_0(z,w)= \frac{1}{(1-z_1\bar
w_1)^{\lambda}(1- z_2\bar w_2)^{\mu}}-1,~ z,w\in\D^2.$$
The set $\{z_1^mz_2^n: m,n\geq 0,(m,n)\neq (0,0)\}$ forms an orthogonal basis for $H_0^{(\lambda, \mu)}(\D^2)$.
Also $\langle z_1^lz_2^k, M_1^*z_1^{m+1}\rangle = \langle z_1^{l+1}z_2^k, z_1^{m+1}\rangle=0$, unless $l=m, k=0$
and $m>0$. In consequence,
$$\langle z_1^m, M_1^*z_1^{m+1}\rangle= \langle z_1^{m+1}, z_1^{m+1}\rangle =\frac{1}{(-1)^{m+1}\binom{-\lambda}{m+1}} =
\frac{(-1)^m\binom{-\lambda}{m}}{(-1)^{m+1}\binom{-\lambda}{m+1}}\langle z_1^m, z_1^m\rangle.$$
Then
$$\langle z_1^lz_2^k, M_1^*z_1^{m+1} -\frac{m+1}{\lambda+m}z_1^m\rangle=0 \mbox{~for~all~} l,k\geq 0, (l,k)\neq
(0,0),$$ where $\binom{-\lambda}{m}= (-1)^{m}\frac{\lambda(\lambda+1)\ldots(\lambda+m-1)}{m!}$. Now, $\langle
z_1^lz_2^k, M_1^*z_1\rangle=\langle z_1^{l+1}z_2^k, z_1\rangle=0$, $l,k\geq 0$ and $(l,k)\neq (0,0)$. Therefore,
we have
$$M_1^*z_1^{m+1} =
\begin{cases} \frac{m+1}{\lambda+m}\,z_1^m & m > 0 \\
0 & m=0.
\end{cases}
$$
Similarly,
$$M_2^*z_2^{n+1} =
\begin{cases} \frac{n+1}{\lambda+n}\,z_1^n & n > 0 \\
0 & n=0.
\end{cases}
$$
We easily verify that $\langle z_1^lz_2^k, M_2^*z_1^{m+1}\rangle = \langle z_1^lz_2^{k+1}, z_1^{m+1}\rangle=0$.
Hence $M_2^*z_1^{m+1}=0=M_1^*z_2^{n+1}$ for $m,n\geq 0$. Finally, calculations similar to the one given above,
show that
$$M_1^*z_1^{m+1}z_2^{n+1}=\frac{m+1}{\lambda+m} z_1^mz_2^{n+1} \mbox{~and~} M_2^*z_1^{m+1}z_2^{n+1}=\frac{n+1}{\mu+n}
z_1^{m+1}z_2^n, m.n \geq 0$$ Therefore we have $$(M_1M_1^*+M_2M_2^*):\left\{%
\begin{array}{ll}
    z_1^{m+1}\longmapsto \frac{m+1}{\lambda+m}z_1^{m+1}, & \hbox{for $m>0$;} \\
    z_2^{n+1}\longmapsto \frac{n+1}{\mu+n}z_2^{n+1}, & \hbox{for $n>0$;} \\
    z_1^{m+1}z_2^{n+1}\longmapsto (\frac{m+1}{\lambda+m}+\frac{n+1}{\mu+n})z_1^{m+1}z_2^{n+1}, & \hbox{for $m,n\geq 0$;} \\
    z_1,z_2\longmapsto 0. & \hbox{ } \\
\end{array}%
\right.    $$ Also, since $D_{\mathbf M^*}f=(M_1^*f,M_2^*f)$, we have  $$D_{\mathbf M^*}:\left\{%
\begin{array}{ll}
    z_1^{m+1}\longmapsto (\frac{m+1}{\lambda+m}z_1^m,0), & \hbox{for $m>0$;} \\
    z_2^{n+1}\longmapsto (0,\frac{n+1}{\mu+n}z_2^n), & \hbox{for $n>0$;} \\
    z_1^{m+1}z_2^{n+1}\longmapsto (\frac{m+1}{\lambda+m}z_1^{m}z_2^{n+1},\frac{n+1}{\mu+n}z_1^{m+1}z_2^n), & \hbox{for $m,n\geq 0$;} \\
    z_1,z_2\longmapsto (0,0). & \hbox{ } \\
\end{array}%
\right.    $$ 
It is easy to calculate $V_{\mathbf M}(0)$ and $Q_{\mathbf M}(0)$ and show that  $$V_{\mathbf M}(0):\left\{%
\begin{array}{ll}
    z_1^{m+1}\longmapsto \sqrt{\frac{m+1}{\lambda+m}}(z_1^m,0), & \hbox{for $m>0$;} \\
    z_2^{n+1}\longmapsto \sqrt{\frac{n+1}{\mu+n}}(0,z_2^n), & \hbox{for $n>0$;} \\
    z_1^{m+1}z_2^{n+1}\longmapsto \frac{1}{\sqrt{\frac{m+1}{\lambda+m}+\frac{n+1}{\mu+n}}}
    (\frac{m+1}{\lambda+m}z_1^{m}z_2^{n+1},\frac{n+1}{\mu+n}z_1^{m+1}z_2^n), & \hbox{for $m,n\geq 0$;} \\
    z_1,z_2\longmapsto (0,0), & \hbox{ } \\
\end{array}%
\right.    $$ while $$Q_{\mathbf M}(0):\left\{%
\begin{array}{ll}
    z_1^{m+1}\longmapsto \frac{1}{\sqrt{\frac{m+1}{\lambda+m}}}z_1^{m+1}, & \hbox{for $m>0$;} \\
    z_2^{n+1}\longmapsto \frac{1}{\sqrt{\frac{n+1}{\mu+n}}}z_2^{n+1}, & \hbox{for $n>0$;} \\
    z_1^{m+1}z_2^{n+1}\longmapsto \frac{1}{\sqrt{\frac{m+1}{\lambda+m}+\frac{n+1}{\mu+n}}}z_1^{m+1}z_2^{n+1},
    & \hbox{for $m,n\geq 0$;} \\
    z_1,z_2\longmapsto 0. & \hbox{ } \\
\end{array}%
\right.   $$ Now for $w \in \Delta (0,\varepsilon)^*$, $$P(\bar w, 0)= (I- R_{\mathbf M}(0)D_{\bar
w})^{-1}P_{\ker D_{M^*}}= \d_{n=0}^{\infty}(R_{\mathbf M}(0)D_{\bar w})^nP_{\ker D_{M^*}},$$ where $R_{\mathbf
M}(0)=Q_{\mathbf M}(0)V_{\mathbf M}(0)^*$. The vectors $z_1$ and $z_2$ forms a basis for ${\ker D_{\mathbf
M^*}}$ and therefore define a holomorphic frame: $\big ( P(\bar w, 0)z_1, P(\bar w, 0)z_2 \big )$. Recall that
$P(\bar w, 0)z_1 = \d_{n=0}^{\infty}(R_{\mathbf M}(0)D_{\bar w})^nz_1$ and $P(\bar w,
0)z_2=\d_{n=0}^{\infty}(R_{\mathbf M}(0)D_{\bar w})^nz_2$. To describe these explicitly,
we calculate $(R_{\mathbf M}(0)D_{\bar w})z_1$ and $(R_{\mathbf M}(0)D_{\bar w})z_2$:
\begin{eqnarray*}(R_{\mathbf M}(0)D_{\bar w})z_1 &=& R_{\mathbf M}(0)(\bar w_1, z_1, \bar w_2z_2)\\ &=& \bar w_1 R_{\mathbf M}(0)(z_1,0)+\bar
w_2R_{\mathbf M}(0)(0, z_2) \\ &=& \bar w_1 Q_{\mathbf M}(0)V_{\mathbf M}(0)^*(z_1,0)+\bar w_2Q_{\mathbf
M}(0)V_{\mathbf M}(0)^*(0, z_2).
\end{eqnarray*}
We see that $$V_{\mathbf M}(0)^*(z_1,0)=\d_{l,k\geq 0,\,(l,k)\neq (0,0)}\langle V_{\mathbf M}(0)^*(z_1,0),
\frac{z_1^lz_2^k}{\parallel z_1^lz_2^k\parallel}\rangle \frac{z_1^lz_2^k}{\parallel z_1^lz_2^k\parallel}.$$
Therefore,
$$\langle V_{\mathbf M}(0)^*(z_1,0), z_1^lz_2^k\rangle= \langle (z_1,0), V_{\mathbf M}(0)(z_1^lz_2^k)\rangle,\, l,k\geq 0,(l,k)\neq (0,0). $$
From the explicit form of $V_{\mathbf M}(0)$, it is clear that the inner product given above is $0$ unless
$l=2,\,k=0$. For $l=2,\,k=0$, we have
$$\langle (z_1,0),
V_{\mathbf M}(0)z_1^2\rangle = \sqrt{\frac{2}{\lambda+1}}\parallel
z_1\parallel^2=\sqrt{\frac{2}{\lambda+1}}\frac{1}{\lambda}.$$ Hence $$V_{\mathbf M}(0)^*(z_1,0) =
\sqrt{\frac{2}{\lambda+1}}\frac{1}{\lambda}\frac{z_1^2}{\parallel z_1^2\parallel^2} =
\sqrt{\frac{2}{\lambda+1}}\frac{1}{\lambda}\frac{\lambda(\lambda+1)}{2}z_1^2 =
\sqrt{\frac{\lambda+1}{2}}z_1^2.$$ Again, to calculate $V_{\mathbf M}(0)^*(0,z_1)$, we note that $\langle
V_{\mathbf M}(0)^*(0,z_1), z_1^lz_2^k\rangle$ is $0$ unless $l=1,\,m=1$. For $l=1,\,m=1$, we have
\begin{eqnarray*}\langle V_{\mathbf M}(0)^*(0,z_1),
z_1z_2\rangle &=& \langle (0,z_1), V_{\mathbf M}(0)z_1z_2\rangle\\ &=&
\langle\frac{1}{\sqrt{\frac{1}{\lambda}+\frac{1}{\mu}}}(\frac{1}{\lambda}z_2, \frac{1}{\mu}z_1), (0,z_1)\rangle \\
&=& \frac{1}{\sqrt{\frac{1}{\lambda}+\frac{1}{\mu}}} \frac{1}{\mu}\parallel z_1\parallel^2
=\frac{1}{\sqrt{\frac{1}{\lambda}+\frac{1}{\mu}}} \frac{1}{\lambda\mu}. \end{eqnarray*} Thus
$$V_{\mathbf M}(0)^*(0,z_1)=\langle V_{\mathbf M}(0)^*(0,z_1), z_1z_2\rangle \frac{z_1z_2}{\parallel z_1z_2\parallel^2}=
\frac{1}{\sqrt{\frac{1}{\lambda}+\frac{1}{\mu}}} z_1z_2.$$ Since
\begin{eqnarray*}
Q_{\mathbf M}(0)z_1^2 &=& \sqrt{\frac{\lambda+1}{2}}z_1^2,\\
Q_{\mathbf M}(0)z_1z_2 &=&\frac{1}{\sqrt{\frac{1}{\lambda}+\frac{1}{\mu}}}z_1z_2,\\
Q_{\mathbf M}(0)z_2^2&=&\sqrt{\frac{\mu+1}{2}}z_2^2,
\end{eqnarray*}
it follows that
$$R_{\mathbf M}(0)D_{\bar w}z_1 = \bar w_1\frac{\lambda+1}{2}z_1^2+\bar w_2\frac{\lambda\mu}{\lambda+\mu}z_1z_2.$$ Similarly, we obtain the formula
$$R_{\mathbf M}(0)D_{\bar w}z_2 = \bar
w_1\frac{\lambda\mu}{\lambda+\mu}z_1z_2+\bar w_2\frac{\mu+1}{2}z_2^2.$$ We claim that
\begin{eqnarray}\label{orth} \langle (R_{\mathbf M}(0)D_{\bar w})^mz_i, (R_{\mathbf M}(0)D_{\bar w})^nz_j\rangle =0 \mbox{~for~ all~} m\neq n
\mbox {~and~} i,j=1,2.\end{eqnarray} This makes the calculation of
$$h(w,w) =\big (\! \big (\langle P(\bar w, 0)z_i,P(\bar w,
0)z_j\rangle \big )\! \big )_{1\leq i,j\leq 2},\, w \in U\subset \mathbb D^2,$$ which is the Hermitian metric
for the vector bundle $\mathcal P$, on some small open set $U\subseteq \mathbb D^2$ around $(0,0)$,
corresponding to the module $H^{(\lambda, \mu)}_0(\mathbb D^2)$, somewhat easier.

We will prove the claim by showing that $(R_{\mathbf M}(0)D_{\bar w})^nz_i$ consists of terms of degree $n+1$.
For this, it is enough to calculate $V_{\mathbf M}(0)^*(z_1^lz_2^k,0)$ and $V_{\mathbf M}(0)^*(0,z_1^lz_2^k)$
for different $l,\,k \geq 0$  such that $(l,k)\neq (0,0)$. Calculations similar to that of $V_{\mathbf M}(0)^*$
show that
\begin{eqnarray*}V_{\mathbf M}(0)^*(z_1^m,0)=
\sqrt{\frac{\lambda+m}{m+1}}z_1^{m+1}, V_{\mathbf M}(0)^*(0, z_2^n)= \sqrt{\frac{\mu+n}{n+1}}z_2^{n+1} \mbox{~and,~}\\
V_{\mathbf M}(0)^*(z_1^mz_2^{n+1},0)=V_{\mathbf M}(0)^*(0,z_1^{m+1}z_2^n)=
\frac{1}{\sqrt{\frac{m+1}{\mu+n}+\frac{n+1}{\mu+n}}}z_1^{m+1}z_2^{n+1}.\end{eqnarray*} Recall that $(R_{\mathbf
M}(0)D_{\bar w})z_i$ is of degree $2$.   From the equations given above, inductively,  we see that $(R_{\mathbf
M}(0)D_{\bar w})^nz_i$ is of degree $n+1$. Since monomials are orthogonal in $H^{(\lambda, \mu)}(\D^2)$, the
proof of claim \eqref{orth} is complete. We then have $$P(\bar w, 0)z_1= z_1+ \bar
w_1\frac{\lambda+1}{2}z_1^2+\bar w_2\frac{\lambda\mu}{\lambda+\mu}z_1z_2 + \d_{n=2}^{\infty}(R_{\mathbf
M}(0)D_{\bar w})^nz_1\mbox{~and~}$$ $$P(\bar w, 0)z_2= z_2+\bar w_1\frac{\lambda\mu}{\lambda+\mu}z_1z_2+\bar
w_2\frac{\mu+1}{2}z_2^2 +\d_{n=2 }^{\infty}(R_{\mathbf M}(0)D_{\bar w})^nz_2.$$ Putting all of this together, we
see that
$$h(w,w) = \begin{pmatrix} \lambda & 0 \\ 0 & \mu \end{pmatrix}
 + \sum a_{IJ} w^I\bar{w}^J,$$
where the sum is over all multi-indices $I,J$ satisfying ${|I|, |J| > 0}$ and $w^I=w_1^{i_1}w_2^{i_2}$,
$\bar{w}^J=\bar{w}_1^{j_1}\bar{w}_2^{j_2}$. The metric $h$ is (almost) normalized at $(0,0)$, that is,
$h(w,0) =\Big ( \begin{smallmatrix} \lambda & 0 \\ 0 & \mu \end{smallmatrix} \Big ).$  The metric $h_0$ obtained
by conjugating the metric $h$ by the invertible (constant) linear transformation $\Big ( \begin{smallmatrix}
\sqrt{\lambda} & 0 \\ 0 & \sqrt{\mu} \end{smallmatrix} \Big )$ induces an equivalence of holomorphic Hermitian
bundles. The vector bundle $\mathcal P$ equipped with the Hermitian metric $h_0$ has the additional property
that the metric is normalized: $h_0(w,0) = I$. The coefficient of $dw_i \wedge d\bar{w}_j$, $i,j=1,2$, in the
curvature of the holomorphic Hermitian bundle $\mathcal P$ at $(0,0)$  is then the Taylor coefficient of
$w_i\,\bar{w}_j$ in the expansion of $h_0$
around $(0,0)$ (cf. \cite[Lemma 2.3]{Wells}).

Thus the normalized metric $h_0(w,w)$, which is real analytic, is of the form
\begin{eqnarray*}h_0(w,w) &=& \left(%
\begin{array}{cc}
  \lambda\langle P(\bar w,0)z_1, P(\bar w, 0)z_1\rangle & \sqrt{\lambda\mu}\langle P(\bar w,0)z_1, P(\bar w, 0)z_2\rangle \\
  \sqrt{\lambda\mu}\langle P(\bar w,0)z_2, P(\bar w, 0)z_1\rangle & \mu\langle P(\bar w,0)z_2, P(\bar w, 0)z_2\rangle \\
\end{array}%
\right)\\ &=& I+\left(%
\begin{array}{cc}
  \frac{\lambda+1}{2}|w_1|^2+\frac{\lambda^2\mu}{(\lambda+\mu)^2}|w_2|^2
   &
  \frac{1}{\sqrt{\lambda\mu}}\big{(}\frac{\lambda\mu}{\lambda+\mu}\big{)}^2w_1\bar w_2 \\
  \frac{1}{\sqrt{\lambda\mu}}\big{(}\frac{\lambda\mu}{\lambda+\mu}\big{)}^2w_2\bar w_1  &
   \frac{\lambda\mu^2}{(\lambda+\mu)^2}|w_1|^2+\frac{\mu+1}{2}|w_2|^2
  \\
\end{array}%
\right)+ O(|w|^3),\end{eqnarray*} where $O(|w|^3)_{i,j}$ is of degree $\geq\, 3$. Explicitly, it is of the form
$$\d_{n=2}^{\infty}\langle (R_{\mathbf M}(0)D_{\bar w})^nz_i,(R_{\mathbf M}(0)D_{\bar w})^nz_j\rangle.$$  The
curvature at $(0,0)$, as pointed out earlier, is given by $\bar\partial\partial h_0(0,0)$. Consequently, if
$H_0^{(\lambda, \mu)}(\D^2)$ and
$H_0^{(\lambda^\prime, \mu^\prime)}(\D^2)$ are equivalent, then  the corresponding holomorphic Hermitian vector bundles $\mathcal P$ and $\tilde{\mathcal P}$ of rank $2$ must be equivalent. Hence their curvatures, in particular, at $(0,0)$, must be unitarily equivalent. The curvature for $\mathcal P$ at $(0,0)$ is  given by the $2\times 2$ matrices $$\left(%
\begin{array}{cc}
  \frac{\lambda+1}{2} & 0 \\
  0 & \frac{\lambda\mu^2}{(\lambda+\mu)^2} \\
\end{array}%
\right),~ \left(%
\begin{array}{cc}
  0 & \frac{1}{\sqrt{\lambda\mu}}\big{(}\frac{\lambda\mu}{\lambda+\mu}\big{)}^2 \\
  0 & 0 \\
\end{array}%
\right), ~\left(%
\begin{array}{cc}
  0 & 0 \\
  \frac{1}{\sqrt{\lambda\mu}}\big{(}\frac{\lambda\mu}{\lambda+\mu}\big{)}^2 & 0 \\
\end{array}%
\right), ~\left(%
\begin{array}{cc}
  \frac{\lambda^2\mu}{(\lambda+\mu)^2} & 0 \\
  0 & \frac{\mu+1}{2} \\
\end{array}%
\right). $$ The curvature for $\tilde{\mathcal P}$ has a similar form with $\lambda^\prime$ and $\mu^\prime$
in place of $\lambda$ and $\mu$ respectively.
All of them are to be simultaneously equivalent by some unitary map. The only unitary that intertwines the $2\times 2$ matrices $$\left(%
\begin{array}{cc}
  0 & \frac{1}{\sqrt{\lambda\mu}}\big{(}\frac{\lambda\mu}{\lambda+\mu}\big{)}^2 \\
  0 & 0 \\
\end{array}%
\right) \mbox{~and~} \left(%
\begin{array}{cc}
  0 & \frac{1}{\sqrt{\lambda'\mu'}}\big{(}\frac{\lambda'\mu'}{\lambda'+\mu'}\big{)}^2 \\
  0 & 0 \\
\end{array}%
\right)$$ is $aI$ with $|a|=1$. Since this fixes the unitary intertwiner, we see that the $2\times 2$ matrices $$\left(%
\begin{array}{cc}
  \frac{\lambda+1}{2} & 0 \\
  0 & \frac{\lambda\mu^2}{(\lambda+\mu)^2} \\
\end{array}%
\right)\mbox{~ and ~} \left(%
\begin{array}{cc}
  \frac{\lambda'+1}{2} & 0 \\
  0 & \frac{\lambda'\mu'^2}{(\lambda'+\mu')^2} \\
\end{array}%
\right)$$ must be equal. Hence we have $\frac{\lambda+1}{2}= \frac{\lambda+1}{2}$, that is $\lambda=\lambda'$.
Consequently, $\frac{\lambda\mu^2}{(\lambda+\mu)^2}=\frac{\lambda'\mu'^2}{(\lambda'+\mu')^2}$ gives
$\frac{\mu^2}{(\lambda+\mu)^2}=\frac{\mu'^2}{(\lambda+\mu')^2}$ and then
$$\mu^2(\lambda^2+2\lambda\mu'+\mu'^2)=\mu'^2(\lambda^2+2\lambda\mu+\mu^2), \mbox{~that~is,~}
(\mu-\mu')\{\lambda^2(\mu+\mu')+2\lambda\mu\mu'\}=0.$$ We then have $\mu=\mu'$. Therefore, $H_0^{(\lambda,
\mu)}(\D^2)$ and $H_0^{(\lambda^\prime, \mu^\prime)}(\D^2)$ are equivalent if and only if
$\lambda=\lambda^\prime$ and $\mu=\mu^\prime$.

\subsection{\sf \sf The $(n,k)$ examples} \label{nk} For a fixed natural number $j$, let  $I_j$ be the polynomial ideal generated by the set  $\{z_1^{n},z_1^{k_j}z_2^{n-k_j}\}$, $k_j\not = 0$. Let $\mathcal
M_j$ be the closure of $I_j$ in the Hardy space $H^2(\mathbb D^2)$. We claim that $\mathcal M_1$ and $\mathcal
M_2$ are inequivalent as Hilbert module unless $k_1=k_2$.  From Lemma \ref{B_1}, it follows that both the
modules $\mathcal M_1$ and $\mathcal M_2$ are in $\mathrm{B}_1(\mathbb D^2 \setminus X)$, where $X:=\{(0,z):|z|
< 1\}\cup \{(z,0):|z| < 1\}$ is the zero set of the ideal $I_j$, $j = 1,2$. However, there is a holomorphic
Hermitian line bundle corresponding to these modules on the projectivization of $\mathbb D^2 \setminus X$  at
$(0,0)$ (cf. \cite[pp. 264]{dmv}). Following the proof of \cite[Theorem 5.1]{dmv}, we see that if these modules
are assumed to be equivalent, then the corresponding line bundles they determine must also be equivalent. This
leads to contradiction unless $k_1\neq k_2$.

Suppose $L:\mathcal M_1\ra \mathcal M_2$ is given to be a unitary module map. Let $K_j$, $j=1,2$, be the
corresponding reproducing kernel. By our assumption, the localizations of the modules, $M_j(w)$ at the point
$w\in\mathbb D^2\setminus X$ are one dimensional and spanned by the corresponding reproducing kernel $K_j$,
$j=1,2$. Since $L$ intertwines module actions, it follows that $M_f^*LK_1(\cdot,w)=\ov{f(w)}LK_1(\cdot,w)$.
Hence,
\begin{eqnarray}\label{1} \label{th1}LK_1(\cdot,w)=\ov{g(w)}K_2(\cdot,w),\mbox{~for~} w\notin X.\end{eqnarray} We conclude that $g$ must be holomorphic on $\mathbb D^2\setminus X$ since both $LK_1(\cdot,w)$ and $K_2(\cdot,w)$ are
anti-holomorphic in $w$. For $j=1,2$, let $E_j$ be the holomorphic line bundle on $\mathbb P^1$ whose section on
the affine chart $U=\{w_1\neq0\}$ is given by \begin{eqnarray*}s_j(\theta)&=&{{\mbox{lim}}_{w\ra 0, \frac{\bar
w_2}{\bar w_1}=\theta}} \frac{K_j(z,w)}{\bar w_1^n}=\frac{z_1^n\bar w_1^n+z_1^{k_j}z_2^{n-k_j}\bar w_1^{k_j}\bar
w_2^{n-k_j}+\mbox{higher~ order~ terms}}{\bar w_1^n}\\
&=& z_1^{n_{}}+\theta^{n-k_j}z_1^{k_j}z_2^{n-k_j}.\end{eqnarray*} Using the ideas from the proof of
\cite[Theorem 5.1]{dmv}, one shows that $|g(w)|$ has a finite limit at each point of the variety $X$.  By the
Riemann removable singularity theorem, it follows that $g$ extends to a holomorphic function on all of $\mathbb
D^2$. Then from \eqref{1}, and the expression of $s_j(\theta)$, by a limiting argument, we find that
$Ls_1(\theta)=g(\theta)s_2(\theta)$.  The unitarity of the map $L$ implies that $$\|L
s_1(\theta)\|^2=|g(\theta)|^2\|s_2(\theta)\|^2$$ and consequently the bundles $E_j$ determined by $\mathcal
M_j$, $j=1,2$, on $\mathbb P^1$ are equivalent. We now calculate the curvature to determine when these line
bundles are equivalent. Since the monomials are orthonormal, we note that the square norm of the section is
given by $${\parallel s_1(\theta)\parallel}^2= 1+|\theta|^{2(n-k_j)}.$$ Consequently the curvature (actually
coefficient of the $(1,1)$ form $d\theta\wedge d\bar\theta$) of the line bundle on the affine chart $U$ is given
by
\begin{eqnarray*}\label{curv}
{\mathcal K_j}(\theta) &=& -\partial_{\theta}\partial_{\bar\theta} { \log}{\parallel
s_1(\theta)\parallel}^2~=~-\partial_{\theta}\partial_{\bar\theta} {\log}(1+|\theta|^{2(n-k_j)})\\&=&~-\partial_{\theta}\frac{(n-k_j)\theta^{(n-k_j)}\bar\theta^{(n-k_j-1)}}{1+|\theta|^{2(n-k_j)}}\\
&=&-\frac{(n-k_j)^2|\theta|^{2(n-k_j-1)}\{1+|\theta|^{2(n-k_j)}\}-
(n-k_j)^2|\theta|^{2(n-k_j)}|\theta|^{2(n-k_j-1)}}{\{1+|\theta|^{2(n-k_j)}\}^2}\\&=&
-\frac{(n-k_j)^2|\theta|^{2(n-k_j-1)}}{\{1+|\theta|^{2(n-k_j)}\}^2}.
\end{eqnarray*}
So if the bundles are equivalent on $\mathbb P^1$, then ${\mathcal K_1}(\theta)={\mathcal K_2}(\theta)$ for
$\theta \in U$, and we obtain
\begin{eqnarray*}\lefteqn{(n-k_1)^2\{|\theta|^{2(n-k_1-1)}+2|\theta|^{2(n-k_2)}|\theta|^{2(n-k_1-1)} +
|\theta|^{4(n-k_2)}|\theta|^{2(n-k_1-1)}\}}\\& - &
(n-k_2)^2\{|\theta|^{2(n-k_2-1)}+2|\theta|^{2(n-k_1)}|\theta|^{2(n-k_2-1)} +
|\theta|^{4(n-k_1)}|\theta|^{2(n-k_2-1)}\} = 0.\\\end{eqnarray*} Since the  equation given above must be
satisfied by all $\theta$ corresponding to the affine chart $U$, it must be an identity. In particular, the
coefficient of $|\theta|^{2\{(n-k_1)+(n-k_2)-1\}}$ must be $0$ implying $(n-k_1)^2=(n-k_2)^2$, that is,
$k_1=k_2$. Hence $\mathcal M_1$ and $\mathcal M_2$ are always inequivalent unless they are equal.

\subsection*{\sf Acknowledgement} The authors would like to thank R. G. Douglas 
and J.-P. Demailly for many hours of very helpful discussions on the topic of this paper.


\end{document}